\documentclass[11pt,reqno]{amsart}

\usepackage{graphicx}
\usepackage{tikz}
\usepackage{capt-of}

\usepackage{xspace}
\usepackage{amsmath,amsthm,amsfonts,amssymb,latexsym,mathrsfs,color,extarrows}
\usepackage{hyperref}

\newtheorem{theorem}{Theorem}
\newtheorem{corollary}[theorem]{Corollary}
\newtheorem{proposition}[theorem]{Proposition}

\newtheorem{lemma}[theorem]{Lemma}
\newtheorem{definition}[theorem]{Definition}
\newtheorem{example}[theorem]{Example}

\newcommand{\SYT}{{\rm SYT\,}}
\newcommand{\mpp}{{\mathcal P}}
\newcommand{\altdes}{{\rm altdes\,}}
\newcommand{\Stab}{{\rm Stab\,}}
\newcommand{\sst}{{\rm st\,}}
\newcommand{\durun}{{\rm durun\,}}
\newcommand{\cpkk}{{\rm cpk\,}}
\newcommand{\cdasc}{{\rm cdasc\,}}
\newcommand{\cddes}{{\rm cddes\,}}
\newcommand{\cddd}{{\rm cdd\,}}
\newcommand{\crun}{{\rm crun\,}}
\newcommand{\rpk}{{\rm rpk\,}}
\newcommand{\altrun}{{\rm altrun\,}}
\newcommand{\las}{{\rm las\,}}
\newcommand{\ass}{{\rm as\,}}
\newcommand{\Sch}{{\rm Sch\,}}
\newcommand{\cval}{{\rm cval\,}}
\newcommand{\cpk}{{\rm cpk\,}}
\newcommand{\st}{{\rm st\,}}
\newcommand{\I}{{\rm I\,}}
\newcommand{\Bddes}{{\rm Bddes\,}}
\newcommand{\Expk}{{\rm Expk\,}}
\newcommand{\Ubdes}{{\rm Ubdes\,}}
\newcommand{\bddes}{{\rm bddes\,}}
\newcommand{\expk}{{\rm expk\,}}
\newcommand{\ubdes}{{\rm ubdes\,}}
\newcommand{\LSPD}{{\rm LSPD\,}}
\newcommand{\LSP}{{\rm LSP\,}}
\newcommand{\LSD}{{\rm LSD\,}}
\newcommand{\LS}{{\rm LS\,}}
\newcommand{\cdd}{{\rm cdd\,}}
\newcommand{\Ddes}{{\rm Ddes\,}}
\newcommand{\Laplat}{{\rm Laplat\,}}
\newcommand{\Orb}{{\rm Orb\,}}
\newcommand{\Des}{{\rm Des\,}}
\newcommand{\Asc}{{\rm Asc\,}}
\newcommand{\Dasc}{{\rm Dasc\,}}
\newcommand{\cda}{{\rm cda\,}}
\newcommand{\JSPD}{{\rm JSPD\,}}
\newcommand{\JSP}{{\rm JSP\,}}
\newcommand{\JS}{{\rm JS\,}}
\newcommand{\el}{{\rm el\,}}
\newcommand{\ol}{{\rm ol\,}}
\newcommand{\m}{{\rm M}}
\newcommand{\doublead}{{\rm doublead\,}}
\newcommand{\udrun}{{\rm udrun\,}}
\newcommand{\run}{{\rm run\,}}
\newcommand{\set}{{\rm Set\,}}
\newcommand{\dasc}{{\rm dasc\,}}
\newcommand{\desp}{{\rm desp\,}}
\newcommand{\fdes}{{\rm fdes\,}}
\newcommand{\fasc}{{\rm fasc\,}}
\newcommand{\bs}{{\mathfrak S}}
\newcommand{\plat}{{\rm plat\,}}
\newcommand{\ap}{{\rm ap\,}}
\newcommand{\lap}{{\rm lap\,}}
\newcommand{\lpk}{{\rm lpk\,}}
\newcommand{\single}{{\rm single\,}}
\newcommand{\pk}{{\rm pk\,}}
\newcommand{\ddes}{{\rm ddes\,}}
\newcommand{\des}{{\rm des\,}}
\newcommand{\Exc}{{\rm Exc\,}}
\newcommand{\exc}{{\rm exc\,}}
\newcommand{\fexc}{{\rm fexc\,}}
\newcommand{\aexc}{{\rm aexc\,}}
\newcommand{\we}{{\rm wexc\,}}
\newcommand{\ce}{{\rm cexc\,}}
\newcommand{\cyc}{{\rm cyc\,}}
\newcommand{\zz}{{\rm z\,}}
\newcommand{\fix}{{\rm fix\,}}
\newcommand{\dc}{{\rm dc\,}}
\newcommand{\DS}{\mathcal{DS}}
\newcommand{\mc}{{\mathcal C}_A}
\newcommand{\cc}{{\mathcal C}}
\newcommand{\man}{\mathfrak{A}_n}
\newcommand{\mdn}{\mathcal{D}}
\newcommand{\mmn}{\mathcal{M}_{2n}}
\newcommand{\D}{\mathfrak{D}}
\newcommand{\msb}{\mathfrak{S}_8}
\newcommand{\msm}{\mathfrak{S}_m}
\newcommand{\mcn}{\mathfrak{C}_{2n}}
\newcommand{\mcnn}{\mathfrak{C}_{2n+1}}
\newcommand{\msn}{\mathfrak{S}_n}
\newcommand{\rss}{\mathcal{SS}}
\newcommand{\ms}{\mathfrak{S}}
\newcommand{\inv}{{\rm inv\,}}
\newcommand{\msnn}{\mathfrak{S}_{n+1}}
\newcommand{\ud}{{\rm ud\,}}
\newcommand{\as}{{\rm as\,}}
\newcommand{\rs}{\mathcal{RS}}
\newcommand{\bn}{{\rm\bf n}}
\newcommand{\lrf}[1]{\lfloor #1\rfloor}
\newcommand{\lrc}[1]{\lceil #1\rceil}
\newcommand{\sgn}{{\rm sgn\,}}
\newcommand{\mbn}{{\mathcal B}_n}
\newcommand{\mq}{\mathcal{Q}}
\newcommand{\mqn}{\mathcal{Q}_n}
\newcommand{\z}{ \mathbb{Z}}
\newcommand{\asc}{{\rm asc\,}}
\newcommand{\tg}{{\rm tg\,}}
\DeclareMathOperator{\N}{\mathbb{N}}
\DeclareMathOperator{\LHP}{LHP}
\DeclareMathOperator{\R}{\mathbb{R}}
\DeclareMathOperator{\Q}{\mathbb{Q}}
\DeclareMathOperator{\re}{Re}
\newcommand{\rz}{{\rm RZ}}
\newcommand{\Eulerian}[2]{\genfrac{<}{>}{0pt}{}{#1}{#2}}
\newcommand{\Stirling}[2]{\genfrac{\{}{\}}{0pt}{}{#1}{#2}}
\newcommand{\stirling}[2]{\genfrac{[}{]}{0pt}{}{#1}{#2}}
\newcommand{\arxiv}[1]{\href{http://arxiv.org/abs/#1}{\texttt{arXiv:#1}}}
\DeclareMathOperator{\YTab}{Ytab}

\title{Derivatives, Eulerian polynomials and the $g$-indexes of Young tableaux}
\author[G.-N.~Han]{Guo-Niu Han}
\address{I.R.M.A., UMR 7501, Universit\'e de Strasbourg et CNRS, 7 rue Ren\'e Descartes, F-67084 Strasbourg, France}
\email{guoniu.han@unistra.fr (G.-N.~Han)}
\author[S.-M.~Ma]{Shi-Mei Ma$^*$}
\thanks{* Corresponding author}
\address{School of Mathematics and Statistics,
        Northeastern University at Qinhuangdao,
         Hebei 066004, P.R. China}
\email{shimeimapapers@163.com (S.-M.~Ma)}
\begin{document}
\begin{abstract}
In this paper we first present summation formulas for $k$-order Eulerian polynomials and
$1/k$-Eulerian polynomials. We then present combinatorial expansions of $(c(x)D)^n$ in terms of inversion sequences as well as $k$-Young tableaux,
where $c(x)$ is a differentiable function in the indeterminate $x$ and $D$ is the derivative with respect to $x$. We define the $g$-indexes of $k$-Young tableaux and Young tableaux, which have important applications in combinatorics. By establishing some relations between $k$-Young tableaux and standard Young tableaux, we express Eulerian polynomials, second-order Eulerian polynomials, Andr\'e polynomials and the generating polynomials of gamma coefficients of Eulerian polynomials in terms of standard Young tableaux,
which imply a deep connection among these polynomials.
\end{abstract}


\maketitle
\section{Introduction}
Let $\msn$ be the symmetric group on the set $[n]=\{1,2,\ldots,n\}$. Let $\pi=\pi(1)\pi(2)\cdots\pi(n)\in\msn$.
A {\it descent} of $\pi$ is an index $i\in [n]$ such that $\pi(i)>\pi(i+1)$ or $i=n$. Let $\des(\pi)$ be the number of descents of $\pi$.
The number $\Eulerian{n}{i}=\{\pi\in\msn: \des(\pi)=i\}$ is called the Eulerian number, and the polynomial
$$A_n(x)=\sum_{\pi\in\msn}x^{\des(\pi)}$$
is called the {\it Eulerian polynomial}.
The historical origin of Eulerian polynomial is the following summation formula (see~\cite{Petersen15}):
\begin{equation}\label{Anx-def}
\left(x\frac{d}{dx}\right)^n\frac{1}{1-x}=\sum_{k=0}^\infty k^nx^k=\frac{A_n(x)}{(1-x)^{n+1}}.
\end{equation}
In the past decades, there has been much work on Eulerian polynomial and its generalizations (see~\cite{Foata09,Hwang20,Rzakowski19,Zhu} for instance). For example,
by using a kind of first-order differential equation,
Rz\c{a}dkowski and Urli\'{n}ska~\cite{Rzakowski19} considered a unified generalization of Eulerian polynomials and second-order Eulerian polynomials.
In the following we first recall the definitions of $k$-order Eulerian polynomials and
$1/k$-Eulerian polynomials, and then we present summation formulas for these polynomials.

A $k$-{\it Stirling permutation} of order $n$ is a permutation of the multiset $\{1^k,2^k,\ldots,n^k\}$ such that for each $i$, $1\leq i\leq n$,
all entries between any two occurrences of $i$ are at least $i$. When $k=2$, the $k$-{\it Stirling permutation} is reduced to the ordinary Stirling permutation (\cite{Gessel78}).
We say that an index $i\in [kn]$ is a {\it descent} of $\sigma$ if
$\sigma_{i}>\sigma_{i+1}$ or $i=kn$. Let $\mqn(k)$ be the set of $k$-Stirling permutations of order $n$. The {\it $k$-order Eulerian polynomials} are defined by $$C_n(x;k)=\sum_{\sigma\in\mqn(k)}x^{\des(\pi)},~C_0(x;k)=1.$$
Following~\cite[Lemma~1]{Dzhumadil14},
the polynomials $C_n(x;k)$ satisfy the recurrence relation
\begin{equation}\label{recuCnx}
C_{n+1}(x;k)=(kn+1)xC_n(x;k)+x(1-x)C_n'(x;k),
\end{equation}
In particular, $C_n(x;1)=A_n(x)$.
When $k=2$, the polynomial $C_n(x;k)$ is reduced to the {\it second-order Eulerian polynomial} $C_n(x)$, i.e., $C_n(x;2)=C_n(x)$.
Stirling permutations and the second-order Eulerian polynomial were defined by Gessel and Stanley~\cite{Gessel78}, and they proved that
$$\sum_{k=0}^\infty \Stirling{k+n}{k}x^k=\frac{C_n(x)}{(1-x)^{2n+1}},$$
where $\Stirling{n}{k}$ is the {\it Stirling number of the second kind}, i.e.,
the number of ways to partition the set $[n]$ into $k$ non-empty subsets. The second-order Eulerian polynomials have been extensively studied in recent years, see~\cite{Haglund12,Hwang20,Ma1902} and references therein.

Let ${\rm s}=\{s_i\}_{i\geq1}$ be a sequence of positive integers.
A geometric interpretation of Eulerian polynomials is obtained by considering the
$\rm{s}$-lecture hall polytope $\mpp_n^{(\rm{s})}$,
which is defined by
$$\mpp_n^{(\rm{s})}=\left\{(\lambda_1,\lambda_2,\ldots,\lambda_n)\in\mathbb{R}^n~\big{|}~ 0\leq\frac{\lambda_1}{s_1}\leq\frac{\lambda_2}{s_2}\leq\cdots \leq\frac{\lambda_n}{s_n}\leq 1\right\}.$$
Set $e_0=0$ and $s_0=1$.
Let ${\rm I}_n^{({\rm s})}=\{\mathbf{e}=(e_1,\ldots,e_n)\in\z^n \mid 0\leq e_i<s_i ~\text{for $1\leq i\leq n$}\}$ be
the set of $n$-dimensional {\it$\rm s$-inversion sequences}.
The polynomial
$$E_n^{(\rm s)}(x)=\sum_{\mathbf{e}\in \rm{I}_n^{(\rm s)}}x^{\asc(\mathbf{e})}$$
is known as the {\it $\rm s$-Eulerian polynomial}, where
$\asc(\mathbf{e})=\#\left\{i\in\{0,1,2,\ldots,n-1\}\big{|}\frac{e_i}{s_i}<\frac{e_{i+1}}{s_{i+1}}\right\}$.
In particular, we have
$$E_n^{(1,2,\ldots,n)}(x)=A_n(x)/x.$$

Let $k$ be a fixed positive integer.
The {\it $1/k$-Eulerian polynomials} $A_{n}^{(k)}(x)$ are defined by the generating function
\begin{equation*}\label{Ankx-deff}
\sum_{n=0}^\infty A_{n}^{(k)}(x)\frac{z^n}{n!}=\left(\frac{1-x}{e^{kz{(x-1)}}-x}\right)^{\frac{1}{k}}.
\end{equation*}
Savage and Viswanathan~\cite{Savage12} showed that
\begin{equation*}\label{Ankx-Savage}
A_n^{(k)}(x)=E_n^{(1,k+1,2k+1\ldots,(n-1)k+1)}(x).
\end{equation*}
For $\pi\in\msn$, an {\it excedance} of $\pi$ is an index $i\in [n]$ such that $\pi(i)>i$.
Let $\exc(\pi)$ (resp.~$\cyc(\pi)$) be the number of excedances (resp.~cycles) of $\pi$.
It follows from~\cite[Proposition~7.3]{Brenti00} that
$$A_{n}^{(k)}(x)=\sum_{\pi\in\msn}x^{\exc(\pi)}k^{n-\cyc(\pi)}.$$
Another combinatorial interpretation of $A_{n}^{(k)}(x)$ is given as follows:
\begin{equation*}\label{Ankx-stirling}
A_n^{(k)}(x)=\sum_{\sigma\in \mqn(k)}x^{\ap(\sigma)},
\end{equation*}
where $\ap(\sigma)$ is the number of the longest ascent plateaus of $\sigma$, i.e., the number of indexes
$i\in \{2,3,\ldots,nk-k+1\}$ such that $\sigma_{i-1}<\sigma_i=\sigma_{i+1}=\cdots=\sigma_{i+k-1}$ (see~\cite[Theorem~2]{Ma15}).
The polynomials $A_n^{(k)}(x)$ satisfy the recurrence relation
\begin{equation}\label{Ankxrecu}
A_{n+1}^{(k)}(x)=(1+knx)A_n^{(k)}(x)+kx(1-x)\frac{d}{dx}A_n^{(k)}(x),
\end{equation}
with the initial conditions $A_0^{(k)}(x)=A_1^{(k)}(x)=1$ (see~\cite[Eq.~(6)]{Ma15}).
Set $M_n(x)=A_n^{(2)}(x)$. Let
$$N_n(x)=\sum_{\sigma\in \mqn}x^{\lap(\sigma)}$$
be the {\it left ascent plateau polynomial},
where $\lap(\sigma)$ is the number of the left ascent plateaux of $\sigma$, i.e., the number of indices
$i\in \{1,2,3,\ldots, 2n-1\}$ such that $\sigma_{i-1}<\sigma_i=\sigma_{i+1}$ and $\sigma(0)=0$ (see~\cite[Theorem~3]{Ma15}).
From~\cite[p.~2]{Ma17}, we see that $N_n(x)=x^nM_n(1/x)=x^nA_n^{(2)}(1/x)$.
Let $B_n(x)$ be the type $B$ Eulerian polynomial.
According to~\cite[Proposition~1]{Ma17}, we have
\begin{align*}
2^nA_n(x)&=\sum_{i=0}^n\binom{n}{i}N_i(x)N_{n-i}(x),~
B_n(x)=\sum_{i=0}^n\binom{n}{i}N_i(x)M_{n-i}(x).
\end{align*}

Let $F_n:=F_n(x;\alpha,\beta,a,b,c)$ be the polynomials defined by the following relation:
\begin{equation*}\label{GeneralEulerian-def}
\left(\frac{a+bx+cx^2}{(1-x)^\alpha}\frac{d}{dx}\right)^n\frac{1}{(1-x)^\beta}=\frac{F_n}{(1-x)^{n+n\alpha+\beta}}.
\end{equation*}
Then $F_0=1$ and it is routine to verify that the polynomials $F_n$ satisfy the recurrence relation
\begin{equation}\label{RecuFn}
F_{n+1}=(n+n\alpha+\beta)(a+bx+cx^2)F_n+(a+bx+cx^2)(1-x)F_n'.
\end{equation}
Comparing~\eqref{recuCnx} and~\eqref{Ankxrecu} with~\eqref{RecuFn}, it is routine to verify the first main result of this paper.
\begin{theorem}\label{thm001}
Let $k$ be a positive integer. For $n\geq 1$, we have
\begin{align*}
&\left(\frac{x}{(1-x)^k}\frac{d}{dx}\right)^n\frac{1}{1-x}=\frac{C_n(x;k+1)}{(1-x)^{n+kn+1}},\\
&\left(kx\frac{d}{dx}\right)^n\frac{1}{(1-x)^{1/k}}=\frac{x^nA_n^{(k)}(1/x)}{(1-x)^{n+\frac{1}{k}}}.
\end{align*}
In particular, we have
\begin{equation}\label{eq:Cnx}
\left(\frac{x}{1-x}\frac{d}{dx}\right)^n\frac{1}{1-x}=\frac{C_n(x)}{(1-x)^{2n+1}},
\end{equation}
\begin{equation*}\label{eq:Nnx}
\left(2x\frac{d}{dx}\right)^n\frac{1}{\sqrt{1-x}}=\frac{N_n(x)}{(1-x)^{n}\sqrt{1-x}}.
\end{equation*}
\end{theorem}

Throughout this paper, we always let $c:=c(x)$ and $f:=f(x)$ be two differentiable functions in the indeterminate $x$, and let $D=\frac{d}{dx}$.
Motivated by Theorem~\ref{thm001}, we shall consider expansions of $(cD)^nf$. The paper is organized as follows.
In the next section, we collect the definitions, notation and preliminary results.
In Section~\ref{section3}, we express $(cD)^nf$ in terms of inversion sequences as well as $k$-Young tableaux. In particular, we define the $g$-indexes of $k$-Young tableau and Young tableau, which have important applications. Also, several main results including Theorems~\ref{th:Cn} and~\ref{th:An} are stated in that section.
In Sections~\ref{section04},~\ref{section05} and~\ref{section06}, we respectively prove three main results, i.e.,
second-order Eulerian polynomials, Eulerian polynomials and Andr\'e polynomials can be expressed in terms of standard Young tableaux.
\section{Preliminary}\label{section2}
The expansions of $(cD)^nf$ have been studied as early as 1823 by Scherk~\cite{Scherk1823}. An illustration of the correspondence between
Scherk's expansion of $(cD)^nf$ and forests of trees can be found in Appendix A of~\cite{Blasiak10}.
In particular, Scherk~\cite[p.~6]{Scherk1823} found that
\begin{equation}\label{Stirling-def}
(xD)^n=\sum_{k=0}^n\Stirling{n}{k}x^kD^k.
\end{equation}
Many generalizations and variations of~\eqref{Stirling-def} frequently appeared in combinatorics and normal ordering problems (see~\cite{Charalambides88,Eu17,Mansour16} for instance).

It will be convenient in the sequel to adopt the convention that $\mathbf{f}_k=D^kf$ and $c_k=D^kc$. In particular, $\mathbf{f}_0=f$ and $c_0=c$.
The first few $(cD)^nf$ are given as follows:
\begin{align*}
(cD)f&=(c)  {\mathbf{f}}_1,\\
(cD)^2f&=(c c_1 )  {\mathbf{f}}_1 +(c^2 )  {\mathbf{f}}_2,\\
(cD)^3f&=(c c_1^2  +c^2 c_2 )  {\mathbf{f}}_1 +(3c^2 c_1 )  {\mathbf{f}}_2 +(c^3 )  {\mathbf{f}}_3,\\
(cD)^4f&=(c c_1^3  +4c^2 c_1 c_2  +c^3 c_3 )  {\mathbf{f}}_1 +(7c^2 c_1^2  +4c^3 c_2 )  {\mathbf{f}}_2 +(6c^3 c_1 )  {\mathbf{f}}_3 +(c^4 )  {\mathbf{f}}_4,\\
(cD)^5f&=(c c_1^4  +11c^2 c_1^2 c_2  +4c^3 c_2^2  +7c^3 c_1 c_3  +c^4 c_4 )  {\mathbf{f}}_1 +(15c^2 c_1^3  +30c^3 c_1 c_2 \\
& \quad +5c^4 c_3 )  {\mathbf{f}}_2 +(25c^3 c_1^2  +10c^4 c_2 )  {\mathbf{f}}_3 +(10c^4 c_1 )  {\mathbf{f}}_4 +(c^5 )  {\mathbf{f}}_5.
\end{align*}
\begingroup\vspace*{-\baselineskip}
\captionof{table}{Expansions of $(cD)^nf$}\label{tab:1}
\vspace*{\baselineskip}\endgroup

For $n\geq 1$,
we define
\begin{equation}\label{Ank-def}
(cD)^nf =\sum_{k=1}^nA_{n,k}\mathbf{f}_k.
\end{equation}
It is evident that $A_{n,k}$ is a function of $c,c_1,\ldots,c_{n-k}$. Thus we can write $A_{n,k}$ as follows: $$A_{n,k}:=A_{n,k}(c,c_1,c_2,\ldots,c_{n-k}).$$
In particular, $A_{1,1}=c$, $A_{2,1}=cc_1$ and $A_{2,2}=c^2$.
By induction, it is easy to verify that
$A_{n+1,1}=cDA_{n,1}$, $A_{n,n}=c^n$ and for $2\leq k\leq n$, we have
\begin{equation}\label{Ank-recu}
A_{n+1,k}=cA_{n,k-1}+cDA_{n,k}.
\end{equation}
The numbers appearing in $A_{n,k}$ as coefficients can be found in~\cite[A139605]{Sloane}.
We refer the reader to~\cite{Briand20,Mansour16,Mijajlovic98} for various results and examples on the expansions of $(cD)^n$.

In 1973, Comtet obtained the following result.
\begin{proposition}[{\cite{Comtet73}}]
Let $A_{n,k}$ be defined by~\eqref{Ank-def}.
For $1\leq k\leq n$, we have
\begin{equation}\label{Ank-Comtet}
A_{n,k}=\frac{c}{k!}\sum (2-k_1)(3-k_1-k_2)\cdots (n-k_1-k_2-\cdots-k_{n-1})\frac{c_{k_1}}{k_1!}\cdots\frac{c_{k_{n-1}}}{k_{n-1}!},
\end{equation}
where the summation is over all sequences $(k_1,k_2,\ldots,k_{n-1})$ of nonnegative integers such that
$k_1+k_2+\cdots+k_{n-1}=n-k$ and $k_1+\cdots+k_j\leq j$ for any $1\leq j\leq n-1$. 
\end{proposition}

The explicit formula~\eqref{Ank-Comtet} provides a method for calculating $(cD)^nf$, see Table~\ref{tab:1}.
However, to obtain the explicit coefficients in Table~\ref{tab:1}, a further step is needed.
In order to state the other expansion formulas for $A_{n,k}$, we
need to introduce several notations on partitions of integers.

A {\it partition} $\lambda=(\lambda_1,\lambda_2,\ldots,\lambda_{\ell})$ is a weakly decreasing sequence of nonnegative integers. Each $\lambda_i$ is called a {\it part} of $\lambda$.
The sum of the parts of a partition $\lambda$ is denoted by $|\lambda|$. If $|\lambda|=n$, then we say that $\lambda$ is a partition of $n$, also written as $\lambda\vdash n$.
We denote by $m_i$ the number of parts equal $i$. By using the multiplicities, we also denote $\lambda$ by $(1^{m_1}2^{m_2}\cdots n^{m_n})$.
The partition with all parts equal to $0$ is the empty partition. The {\it length} of $\lambda$, denoted $\ell(\lambda)$, is the maximum subscript $j$ such that $\lambda_j>0$.
The {\it Ferrers diagram} of $\lambda$ is graphical representation of $\lambda$ with $\lambda_i$ boxes in its $i$th row and the boxes are left-justified.
For a Ferrers diagram $\lambda\vdash n$ (we will often identify a partition with its Ferrers
diagram), a {\it (standard) Young tableau} (SYT, for short) of shape $\lambda$ is a filling of the $n$ boxes of $\lambda$ with the integers $1,2,\ldots, n$ such that each number is used, and all rows and columns are increasing (from left to right, and from bottom to top, respectively).
Given a Young tableau, we number its rows starting from the bottom and going above.
Let $\SYT(n)$ be the set of standard Young tableaux of size $n$.

For a partition $\lambda=(\lambda_1,\lambda_2,\ldots,\lambda_{\ell})$,
we define $$c_\lambda=\prod_{i=1}^{\ell}c_{\lambda_i},~c_\emptyset=1.$$
Let $\stirling{n}{k}=\#\{\pi\in\msn: \cyc(\pi)=k\}$ be the Stirling numbers of the first kind.
We now recall another expansion formula for $A_{n,k}$.
\begin{proposition}[{\cite{Benkart13,Briand20}}]\label{prop02}
Let $A_{n,k}$ be defined by~\eqref{Ank-def}.
For $n\geq 1$, there exist positive integers $a(n,\lambda)$ such that
\begin{equation}
A_{n,k}=\sum_{\lambda\vdash n-k} a(n,\lambda)c^{n-\ell(\lambda)}c_{\lambda},
\end{equation}
where $\lambda$ runs over all partitions of $n-k$. In particular, we have
\begin{align*}
&\sum_{\lambda\vdash n-k} a(n,\lambda)=\stirling{n}{k},~
a(n,1^{n-k})=\Stirling{n}{k},~
\sum_{\ell(\lambda)=n-k} a(n,\lambda)=\Eulerian{n}{k}.
\end{align*}
\end{proposition}

Motivated by Proposition~\ref{prop02},
in the next section we present the other main results of this paper. More importantly, we define the $g$-indexes of $k$-Young tableau and Young tableau.
\section{Inversion sequences and the $g$-index of Young tableau}\label{section3}
\subsection{Derivatives and inversion sequences}
\hspace*{\parindent}

An integer sequence $\mathbf{e}=(e_1, e_2, \ldots, e_n)$ is an {\it inversion sequence} of length $n$ if
$0\leq e_i<i$ for all $1\leq i\leq n$. Let $\I_n$
be the set of inversion sequences of length $n$.
There is a natural bijection $\psi$ between $\I_n$ and $\msn$ defined by
$\psi(\pi)=\mathbf{e}$,
where $e_i=\#\{j\mid 1\leq j<i~\text{and}~\pi(j)>\pi(i)\}$.

\begin{definition}
For $\mathbf{e}\in\I_n$, let $|\mathbf{e}|_j=\#\{i\mid e_i=j,\ 1\leq i\leq n\}$. Then
	we define $$\phi(\mathbf{e})=c \cdot c_{|\mathbf{e}|_1 } c_{|\mathbf{e}|_2} \cdots  c_{|\mathbf{e}|_{n-1}}\cdot  {\bf f}_{|\mathbf{e}|_0}.$$
\end{definition}

For example, take $n=9$ and $\mathbf{e}=(0,0,1,0,4,2,4,0,1)$, then
$|\mathbf{e}|_0=4, |\mathbf{e}|_1=2, |\mathbf{e}|_2=1, |\mathbf{e}|_3=0, |\mathbf{e}|_4=2$ and $|\mathbf{e}|_j=0$ for $5\leq j\leq 8$.
So that $\phi(\mathbf{e})=c \cdot  c_2 c_1 c c_2 c c c c \cdot f_4 = c^6 c_1 c_2^2  \cdot {\bf f}_4$.

We now present the second main result of this paper.
\begin{theorem}\label{thm01}
For $n\geq1$, we have
\begin{equation}\label{mainthm01}
	(cD)^nf = \sum_{\mathbf{e}\in \I_n} \phi(\mathbf{e}).
	\end{equation}
\end{theorem}
\begin{proof}
When $n=1$, we have $\I_1=\{0\}$ and $\phi(0)=c{\bf f}_1$.
When $n=2$, we have $\I_2=\{00,01\}$. Note that $\phi(00)=c \cdot c  \cdot  f_2$ and $\phi(01)=c \cdot c_1  \cdot  {\bf f}_1$.
Hence~\eqref{mainthm01} is valid for $n=1,2$.
Assume that~\eqref{mainthm01} holds for $n$. Let $\I_{n,k}=\{\mathbf{e}\in\I_n: |\mathbf{e}|_0=k\}$.
Then for any $\mathbf{e}\in \I_{n,k}$, we have
	$$\phi(\mathbf{e})=c \cdot c_{|\mathbf{e}|_1 } \cdot c_{|\mathbf{e}|_2} \cdots  c_{|\mathbf{e}|_{n-1}}\cdot  {\bf f}_k.$$
Let $\mathbf{e}'$ be obtained from $\mathbf{e}=(e_1,e_2,\ldots,e_n)$ by appending $e_{n+1}$. We
distinguish three cases:
\begin{itemize}
	\item [\rm ($i$)] If $e_{n+1}=0$, then $\phi(\mathbf{e'})=c \cdot c_{|\mathbf{e}|_1 } \cdot c_{|\mathbf{e}|_2} \cdots  c_{|\mathbf{e}|_{n-1}}\cdot c\cdot  {\bf f}_{k+1}$;
	\item [\rm ($ii$)]  If $e_{n+1}=i$ and $1\leq i\leq n-1$, then $\phi(\mathbf{e'})=c\cdot  c_{|\mathbf{e}|_1 } \cdot c_{|\mathbf{e}|_2}  \cdots c_{|\mathbf{e}|_{i}+1} \cdots  c_{|\mathbf{e}|_{n-1}}\cdot c\cdot {\bf f}_k$;
	\item [\rm ($iii$)]  If $e_{n+1}=n$, then $\phi(\mathbf{e'})=c\cdot c_{|\mathbf{e}|_1 } \cdot c_{|\mathbf{e}|_2} \cdots  c_{|\mathbf{e}|_{n-1}}\cdot c_1\cdot {\bf f}_k$.
\end{itemize}
It is routine to check that the first case accounts for the term $cA_{n,k-1}$ and the last two cases account for the term $cDA_{n,k}$.
Then $\sum_{\mathbf{e}\in I_{n+1,k}} \phi(\mathbf{e}) = (cA_{n,k-1} + cDA_{n,k}) \mathbf{f}_k= A_{n+1,k}\mathbf{f}_k$, which follows from~\eqref{Ank-recu}.
	This completes the proof.
\end{proof}

\begin{example}
When $n=3$, the correspondence between $\mathbf{e}\in\I_3$ and $\phi(\mathbf{e})$ is illustrated as follows:
\begin{equation*}
\begin{matrix}
	\mathbf{e} & 000 & 001 & 002 & 010 & 011 & 012   \\
|\mathbf{e}|_0~|\mathbf{e}|_1~|\mathbf{e}|_2& 300 & 210 & 201 & 210 & 120 & 111   \\
	\phi(\mathbf{e})& ccc {\bf f}_3 & cc_1c {\bf f}_2& ccc_1 {\bf f}_2 & cc_1c{\bf f}_2 & c c_2c  {\bf f}_1 & cc_1c_1 {\bf f}_1
\end{matrix}
\end{equation*}
So that
$$
	\sum_{\mathbf{e}\in \I_3} \phi(\mathbf{e})=(cc_1^2 +c^2c_2){\bf f}_1 + (3c^2c_1) {\bf f}_2 + c^3 {\bf f}_3.
$$
\end{example}

\begin{example}
When $c=x$, we have $c_0=x, c_1=Dx=1$ and $c_i=0$ for $i\geq 2$.
Then $\phi(\mathbf{e})\neq 0$ unless $|\mathbf{e}|_i=0$ or
$|\mathbf{e}|_i=1$ for all $i\geq 1$.
In this case,
let $k=|\mathbf{e}|_0$,
then $$n-k= \#\{j : |\mathbf{e}|_j=1, 1\leq j \leq n-1\},~(n-1)-(n-k)=\#\{j : |\mathbf{e}|_j=0, 1\leq j \leq n-1\}.$$
Thus
	$\phi(\mathbf{e})=c c_{|\mathbf{e}|_1} \cdots c_{|\mathbf{e}|_{n-1}} {\bf f}_{|\mathbf{e}|_0}=x^{k}{\bf f}_{k}$.
It follows from~\eqref{Stirling-def} that
$(xD)^nf=\sum_{k=0}^{n} \Stirling{n}{k} x^{k}{\bf f}_{k}$.
Hence
\begin{equation*}\label{Stirling}
\Stirling{n}{k}=\#\{\mathbf{e}\in \I_n : |\mathbf{e}|_0=k, |\mathbf{e}|_j=0~\text{or}~1 \text{\ for any $1\leq j\leq n-1$}\}.
\end{equation*}
\end{example}

We can derive Comtet's formula~\eqref{Ank-Comtet} by using Theorem~\ref{thm01}. For $\mathbf{e}\in \I_n$, let
$k=|\mathbf{e}|_0$ and $k_i=|\mathbf{e}|_{n-i}$ for $1\leq i\leq n-1$. Note that
$$k_1+k_2+\cdots+k_{n-1}=n-k$$ and $k_1+\cdots +k_j\leq j$ for each $j$.
Therefore, the number of such $\mathbf{e}$ is equal to
\begin{align*}
	&\binom{1}{k_1}	\binom{2-k_1}{k_2} \binom{3-k_1-k_2} {k_3} \cdots
	\binom{n-k_1-k_2-\cdots -k_{n-1}}{k} \\
	&=\frac{(2-k_1)(3-k_1-k_2)\cdots(n-k_1-k_2-\cdots -k_{n-1})}{k! k_1! k_2! \cdots k_{n-1}!}.
\end{align*}
\subsection{Derivatives and $k$-Young tableaux}
\hspace*{\parindent}

Since the $c_{k_1}$, $c_{k_2}, \ldots, c_{k_{n-1}}$ are commutative, we have to group the
terms in \eqref{Ank-Comtet} which produce the same product $c_{k_1}c_{k_2}\cdots c_{k_{n-1}}$.
We say that the {\it type} of $n$ is a pair $(k, \mu)$, denoted by $(k, \mu)\vdash n$, where $k\in [n]$ and
$\mu=(\mu_1, \ldots,\mu_{n-1})$ is
a partition of $n-k$, i.e., $\mu$ is written up to $n-1$ terms by appending $0$'s
at the end. Let $(k, \mu)$ be a type of $n$. We define
$$\set(\mu)=\{\mu_j \mid 1\leq j \leq n-1\},~|\mu|_j=\#\{i\mid \mu_i=j, 1\leq i\leq n-1\}.$$

Let $(|\mathbf{e}|_0, \mu(\mathbf{e}))$ be the {\it type} of $\mathbf{e}\in \I_n$, where $\mu(\mathbf{e})$ is the decreasing order of $|\mathbf{e}|_1, \ldots, |\mathbf{e}|_{n-1}$.
For each type $(k,\mu)$  of $n$,
let $p_{k, \mu}$ be the number of inversion sequences of type $(k,\mu)$. It follows from Theorem~\ref{thm01}
that
\begin{equation}\label{eq:cD:p}
	(cD)^n f= \sum_{(k, \mu)\vdash n} p_{k, \mu} c c_{\mu_1}c_{\mu_2}\cdots c_{\mu_{n-1}}  \mathbf{f}_k,
\end{equation}
where the summation is taken over all types $(k, \mu)$ of $n$.

\begin{example}
For $1\leq n=k+|\mu|\leq 3$, the numbers $p_{k, \mu}$ are $p_{1,(0)} = 1$, $p_{2,(0)} = 1$, $p_{1,(1)} = 1$,
$p_{3,(0,0)} = 1$, $p_{2,(1,0)} = 3$, ~$p_{1,(2,0)} = 1$ and $p_{1,(1,1)} = 1$.
\end{example}

\begin{lemma}\label{lemma:pml}
By convention, set $p_{0, \mu}=0$.
If $(k, \mu)=(1, (1,1,\ldots, 1))$, then let $p_{k, \mu}=1$.
For other type $(k, \mu)$ of $n$, we have
\begin{equation}	\label{eq:p}
p_{k,\mu}= \sum_{j \in \set(\mu)\setminus\{0\}} (|\mu|_{j-1}+1) p_{k, \mu^{(j)}} + p_{k-1, \mu^{(0)}},
\end{equation}	
where $\mu^{(j)}$ is obtained from $\mu$ by replacing the last occurrences of the part $j$ by $j-1$ and by deleting the last $0$ and
$\mu^{(0)}$ is obtained from $\mu$ by  deleting the last $0$. Thus $(k, \mu^{(j)})\vdash (n-1)$ and $(k-1, \mu^{(0)})\vdash (n-1)$.
\end{lemma}
\begin{proof}
Take an inversion sequence $\mathbf{e}\in \I_n$ of type $(k, \mu)$.
Let $\mathbf{e}'  =(e_1,e_2,\ldots,e_{n-1})\in \I_{n-1}$ be obtained from $\mathbf{e}$ by deleting the last $e_{n}$.
	If $e_n=0$, then, the type of $\mathbf{e}'$ is $(k-1, \mu^{(0)})$.
	This operation is reversible.
	If $e_n=i$ ($1\leq i\leq n-1$) and $|\mathbf{e}|_i=j\in \set(\mu)\setminus\{0\}$, then the type of $\mathbf{e}'$ is $(k, \mu^{(j)})$.
	In this case, the operation is not reversible. We have exactly $(|\mu|_{j-1}+1)$ ways to do the inverses.
	In fact we can append $e_n=i'\not=i$ at the end of $\mathbf{e}'$ with the condition of $|\mathbf{e}|_i-1=j-1=|\mathbf{e}|_{i'}$ to obtain an inversion sequence in $\I_n$ of type $(k, \mu)$.
\end{proof}

As an illustration of~\eqref{eq:p}, in order to get inversion sequences of type $(k,\mu)=(3,(2,1,1,0,0,0))$, we distinguish three cases:
\begin{itemize}
	\item [(i)] For each $\mathbf{e}\in \I_6$ that counted by $p_{2, (2,1,1,0,0)}$, we can get exactly one inversion sequence of type $(k,\mu)$ by appending $e_7=0$ at the end of $\mathbf{e}$;
	\item [(ii)] Let $\mathbf{e}\in \I_6$ be an inversion sequence counted by $p_{3, (1,1,1,0,0)}$. If $|\mathbf{e}|_{i}=1$
then we can append $e_7=i$ at the end of $\mathbf{e}$. As we have three choices for $i$, we get the term  $3 p_{3, (1,1,1,0,0)}$;
	\item [(iii)] Let $\mathbf{e}\in \I_6$ be an inversion sequence counted by $p_{3, (2,1,0,0,0)}$.
If $|\mathbf{e}|_{i}=0$ or $i=6$
then we can append $e_7=i$ at the end of $\mathbf{e}$. As we have four choices for $i$, we get the term  $ 4 p_{3, (2,1,0,0,0)}$.
\end{itemize}
Repeatedly, it is routine to verify that
$$
p_{3,(2,1,1,0,0,0)} = 4 p_{3, (2,1,0,0,0)} + 3 p_{3, (1,1,1,0,0)} + p_{2, (2,1,1,0,0)}
=4\times 120 + 3 \times 90 + 146= 896.$$

Each type $(k, \mu)$ of $n$ can be represented by a picture which contains
$k$ boxes in the bottom row, and the Young diagram of the partition $\mu$ in
the top. Such picture is called a {\it $(k,\mu)$-diagram}. See Figure \ref{fig:Z:diag} (left diagram).
\begin{definition}
Let $(k, \mu)$ be a type of $n$.
A {\it $k$-Young tableau} $Z$ of shape $(k,\mu)$ is a filling
	of the $n$ boxes of the $(k, \mu)$-diagram by the integers $1,2,\ldots,n$ such that
(i) each number is used,
(ii)
all rows and columns in the top Young diagram are increasing (from left to right, and from bottom to top, respectively),
	(iii)	
the bottom row becomes an increasing sequence of lenght $k$, starting with $1$.
\end{definition}
The filling of the
	top Young diagram of the partition $\mu$ is called the {\it top Young tableau} of the $k$-Young tableau.
Unlike the ordinary Young tableau, there is no condition between the bottom row and the top Young tableau.
We always put a special column of $n$ boxes at the left of $k$-Young tableaux, and labelled by the integers $1,2,\ldots,n$  from bottom to top.
See Figure \ref{fig:Z:diag} (right diagram) for an example.
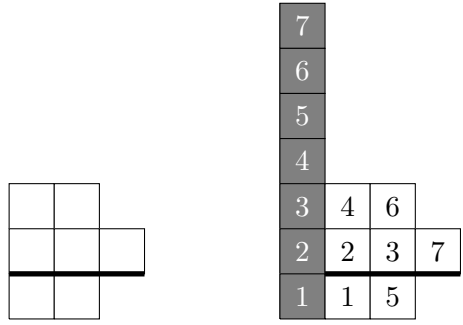
\begin{figure}[tbp]
\begin{center}
\begin{tikzpicture} 
\fill [gray!0](0.0000,0.6000)--(0.0000,1.2000)--(0.6000,1.2000)--(0.6000,0.6000)--(0.0000,0.6000);
\draw [gray!200](0.0000,0.6000)--(0.0000,1.2000)--(0.6000,1.2000)--(0.6000,0.6000)--(0.0000,0.6000);
\draw [gray!200] (0.3000, 0.9000) node [] {};
\fill [gray!0](0.6000,0.6000)--(0.6000,1.2000)--(1.2000,1.2000)--(1.2000,0.6000)--(0.6000,0.6000);
\draw [gray!200](0.6000,0.6000)--(0.6000,1.2000)--(1.2000,1.2000)--(1.2000,0.6000)--(0.6000,0.6000);
\draw [gray!200] (0.9000, 0.9000) node [] {};
None
\fill [gray!0](0.0000,1.8000)--(0.0000,2.4000)--(0.6000,2.4000)--(0.6000,1.8000)--(0.0000,1.8000);
\draw [gray!200](0.0000,1.8000)--(0.0000,2.4000)--(0.6000,2.4000)--(0.6000,1.8000)--(0.0000,1.8000);
\draw [gray!200] (0.3000, 2.1000) node [] {};
\fill [gray!0](0.6000,1.8000)--(0.6000,2.4000)--(1.2000,2.4000)--(1.2000,1.8000)--(0.6000,1.8000);
\draw [gray!200](0.6000,1.8000)--(0.6000,2.4000)--(1.2000,2.4000)--(1.2000,1.8000)--(0.6000,1.8000);
\draw [gray!200] (0.9000, 2.1000) node [] {};
\fill [gray!0](0.0000,1.2000)--(0.0000,1.8000)--(0.6000,1.8000)--(0.6000,1.2000)--(0.0000,1.2000);
\draw [gray!200](0.0000,1.2000)--(0.0000,1.8000)--(0.6000,1.8000)--(0.6000,1.2000)--(0.0000,1.2000);
\draw [gray!200] (0.3000, 1.5000) node [] {};
\fill [gray!0](0.6000,1.2000)--(0.6000,1.8000)--(1.2000,1.8000)--(1.2000,1.2000)--(0.6000,1.2000);
\draw [gray!200](0.6000,1.2000)--(0.6000,1.8000)--(1.2000,1.8000)--(1.2000,1.2000)--(0.6000,1.2000);
\draw [gray!200] (0.9000, 1.5000) node [] {};
\fill [gray!0](1.2000,1.2000)--(1.2000,1.8000)--(1.8000,1.8000)--(1.8000,1.2000)--(1.2000,1.2000);
\draw [gray!200](1.2000,1.2000)--(1.2000,1.8000)--(1.8000,1.8000)--(1.8000,1.2000)--(1.2000,1.2000);
\draw [gray!200] (1.5000, 1.5000) node [] {};
None
\draw [gray!200, line width=2pt](0.0000,1.2000)--(1.8000,1.2000);
None
\fill [gray!100](3.6000,4.2000)--(3.6000,4.8000)--(4.2000,4.8000)--(4.2000,4.2000)--(3.6000,4.2000);
\draw [gray!200](3.6000,4.2000)--(3.6000,4.8000)--(4.2000,4.8000)--(4.2000,4.2000)--(3.6000,4.2000);
\draw [gray!0] (3.9000, 4.5000) node [] {7};
\fill [gray!100](3.6000,3.6000)--(3.6000,4.2000)--(4.2000,4.2000)--(4.2000,3.6000)--(3.6000,3.6000);
\draw [gray!200](3.6000,3.6000)--(3.6000,4.2000)--(4.2000,4.2000)--(4.2000,3.6000)--(3.6000,3.6000);
\draw [gray!0] (3.9000, 3.9000) node [] {6};
\fill [gray!100](3.6000,3.0000)--(3.6000,3.6000)--(4.2000,3.6000)--(4.2000,3.0000)--(3.6000,3.0000);
\draw [gray!200](3.6000,3.0000)--(3.6000,3.6000)--(4.2000,3.6000)--(4.2000,3.0000)--(3.6000,3.0000);
\draw [gray!0] (3.9000, 3.3000) node [] {5};
\fill [gray!100](3.6000,2.4000)--(3.6000,3.0000)--(4.2000,3.0000)--(4.2000,2.4000)--(3.6000,2.4000);
\draw [gray!200](3.6000,2.4000)--(3.6000,3.0000)--(4.2000,3.0000)--(4.2000,2.4000)--(3.6000,2.4000);
\draw [gray!0] (3.9000, 2.7000) node [] {4};
\fill [gray!100](3.6000,1.8000)--(3.6000,2.4000)--(4.2000,2.4000)--(4.2000,1.8000)--(3.6000,1.8000);
\draw [gray!200](3.6000,1.8000)--(3.6000,2.4000)--(4.2000,2.4000)--(4.2000,1.8000)--(3.6000,1.8000);
\draw [gray!0] (3.9000, 2.1000) node [] {3};
\fill [gray!100](3.6000,1.2000)--(3.6000,1.8000)--(4.2000,1.8000)--(4.2000,1.2000)--(3.6000,1.2000);
\draw [gray!200](3.6000,1.2000)--(3.6000,1.8000)--(4.2000,1.8000)--(4.2000,1.2000)--(3.6000,1.2000);
\draw [gray!0] (3.9000, 1.5000) node [] {2};
\fill [gray!100](3.6000,0.6000)--(3.6000,1.2000)--(4.2000,1.2000)--(4.2000,0.6000)--(3.6000,0.6000);
\draw [gray!200](3.6000,0.6000)--(3.6000,1.2000)--(4.2000,1.2000)--(4.2000,0.6000)--(3.6000,0.6000);
\draw [gray!0] (3.9000, 0.9000) node [] {1};
None
\fill [gray!0](4.2000,0.6000)--(4.2000,1.2000)--(4.8000,1.2000)--(4.8000,0.6000)--(4.2000,0.6000);
\draw [gray!200](4.2000,0.6000)--(4.2000,1.2000)--(4.8000,1.2000)--(4.8000,0.6000)--(4.2000,0.6000);
\draw [gray!200] (4.5000, 0.9000) node [] {1};
\fill [gray!0](4.8000,0.6000)--(4.8000,1.2000)--(5.4000,1.2000)--(5.4000,0.6000)--(4.8000,0.6000);
\draw [gray!200](4.8000,0.6000)--(4.8000,1.2000)--(5.4000,1.2000)--(5.4000,0.6000)--(4.8000,0.6000);
\draw [gray!200] (5.1000, 0.9000) node [] {5};
None
\fill [gray!0](4.2000,1.8000)--(4.2000,2.4000)--(4.8000,2.4000)--(4.8000,1.8000)--(4.2000,1.8000);
\draw [gray!200](4.2000,1.8000)--(4.2000,2.4000)--(4.8000,2.4000)--(4.8000,1.8000)--(4.2000,1.8000);
\draw [gray!200] (4.5000, 2.1000) node [] {4};
\fill [gray!0](4.8000,1.8000)--(4.8000,2.4000)--(5.4000,2.4000)--(5.4000,1.8000)--(4.8000,1.8000);
\draw [gray!200](4.8000,1.8000)--(4.8000,2.4000)--(5.4000,2.4000)--(5.4000,1.8000)--(4.8000,1.8000);
\draw [gray!200] (5.1000, 2.1000) node [] {6};
\fill [gray!0](4.2000,1.2000)--(4.2000,1.8000)--(4.8000,1.8000)--(4.8000,1.2000)--(4.2000,1.2000);
\draw [gray!200](4.2000,1.2000)--(4.2000,1.8000)--(4.8000,1.8000)--(4.8000,1.2000)--(4.2000,1.2000);
\draw [gray!200] (4.5000, 1.5000) node [] {2};
\fill [gray!0](4.8000,1.2000)--(4.8000,1.8000)--(5.4000,1.8000)--(5.4000,1.2000)--(4.8000,1.2000);
\draw [gray!200](4.8000,1.2000)--(4.8000,1.8000)--(5.4000,1.8000)--(5.4000,1.2000)--(4.8000,1.2000);
\draw [gray!200] (5.1000, 1.5000) node [] {3};
\fill [gray!0](5.4000,1.2000)--(5.4000,1.8000)--(6.0000,1.8000)--(6.0000,1.2000)--(5.4000,1.2000);
\draw [gray!200](5.4000,1.2000)--(5.4000,1.8000)--(6.0000,1.8000)--(6.0000,1.2000)--(5.4000,1.2000);
\draw [gray!200] (5.7000, 1.5000) node [] {7};
None
\draw [gray!200, line width=2pt](4.2000,1.2000)--(6.0000,1.2000);
None
\end{tikzpicture} 
\end{center}
\caption{$(k=2, \mu=(3,2,0,0,0,0))$-diagram and $k$-Young tableau of shape $(k,\mu)$}
\label{fig:Z:diag}
\end{figure}

\begin{definition}\label{def-g-index01}
Let $Z$ be a $k$-Young tableau of shape $(k,\mu)$, where $k+|\mu|=n$.
For each $v\in [n]$, suppose that $v$ is in the box $(i,j)$
of the top Young diagram,
we define the {\it $g$-index} of $v$, denoted by $g_{Z}(v)$, to be the number of boxes $(i-1,j')$  such that $j'\geq j$ and the letter in this box is less than or equal to $v$ (see Figure \ref{fig:Ytab}, right diagram).
If $v$ is in the bottom row, then we define $g_Z(v)=1$.
The $g$-index of $Z$ is given by $G_Z=g_Z(1) g_Z(2) \cdots g_Z(n)$.
\end{definition}

For the $k$-Young tableau given in Figure~\ref{fig:Z:diag} (right diagram), we have
$$
g_Z(1)=1, \
g_Z(2)=1, \
g_Z(3)=1, \
g_Z(4)=2, \
g_Z(5)=1, \
g_Z(6)=1, \
g_Z(7)=2.
$$

The third main result of this paper is given as follows.
\begin{theorem}\label{th:kmu:tableau}
If $(k, \mu)\vdash n$, then we have
	\begin{equation}	\label{eq:kmu:tableau}
	p_{k,\mu}
		= \sum_{Z} G_Z
	\end{equation}	
	where the summation is taken over all $k$-Young tableaux of shape $(k,\mu)$.
\end{theorem}
\begin{proof}
	Identity \eqref{eq:kmu:tableau} is obtained from Lemma \ref{lemma:pml} by induction on $n$. The maximum letter $n$ in the $k$-Young tableaux $Z$ can be
	at the end of the bottom row, or a corner in the top Young tableau of $Z$.
	In the first case, $g_Z(n)=1$, and removing the letter $n$ yields a
	$(k-1)$-Young tableau of shape $(k-1,\mu)$. In the second case, $g_Z(n)=|\mu|_{j-1}+1$, and removing the letter $n$ yields
	a $k$-Young tableaux of shape $(k, \mu^{(j)})$, where $j$ is the length of the row contained $n$.
	We recover all terms in \eqref{eq:p}.
\end{proof}

The Stirling numbers of the first kind $\stirling{n}{k}$ can be defined as follows:
$$\sum_{k=1}^n\stirling{n}{k}x^k=x(x+1)\cdots (x+n-1).$$
According to~\cite[Proposition~A. 2]{Blasiak10}, we have
$(e^xD)^nf=e^{nx}\sum_{k=1}^n\stirling{n}{k}\mathbf{f}_k$.
We replace $c=e^x$ and $c_j=e^x$ in~\eqref{eq:cD:p}. By Theorem \ref{th:kmu:tableau}, we obtain
\begin{equation*}
	(e^xD)^nf = e^{nx}\sum_{(k,\mu)\vdash n} \sum_{Z} G_Z \mathbf{f}_k,
\end{equation*}
Hence
\begin{equation}
\sum_{(k,\mu)\vdash n} \sum_{Z} G_Z x^k=x(x+1)(x+2)\cdots(x+n-1),
\end{equation}
where the first summation is taken over all type $(k,\mu)$ of $n$, and
the second summation is taken over all $k$-Young tableaux of shape $(k,\mu)$.
For example, when $n=4$, the $k$-Young tableaux with their $g$-indexes are listed in
Figure \ref{fig:kmu4}.
\begin{figure}[tbp]
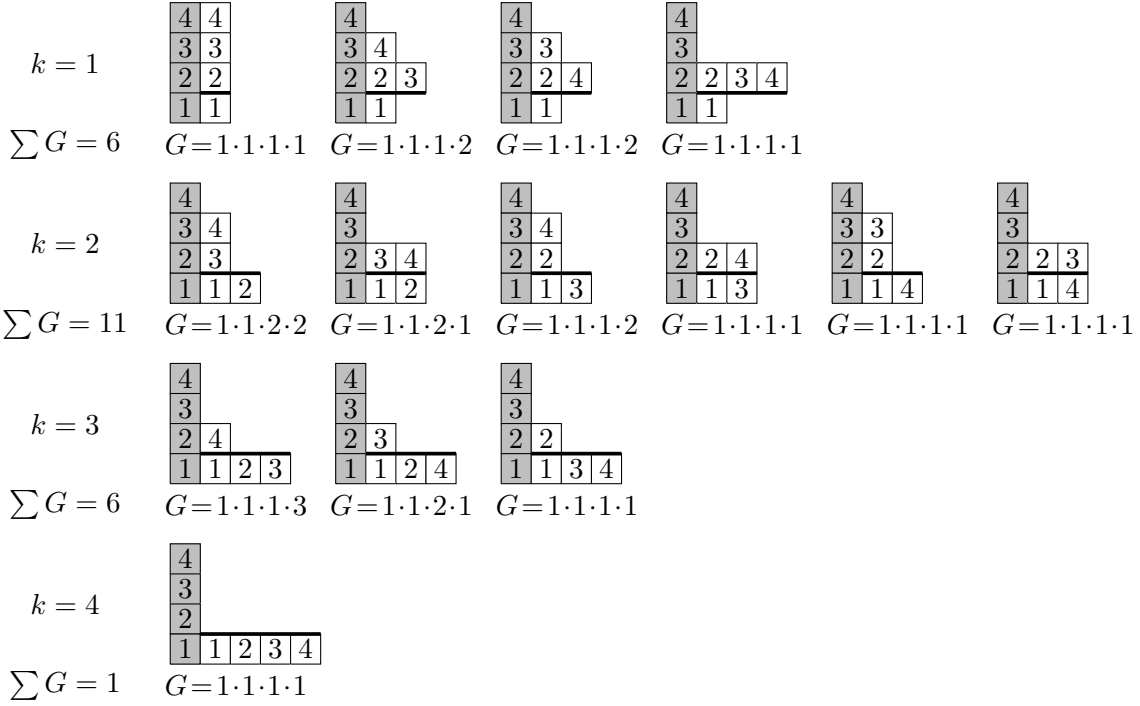

\begin{center}
\end{center}
\caption{All $k$-Young tableaux of size $4$ and their $g$-indexes}
\label{fig:kmu4}
\end{figure}

As an application of Theorem~\ref{th:kmu:tableau}, we give the following result.
\begin{proposition}
Let $\Stirling{n}{k}$ be the Stirling numbers of the second kind. Then we have
$$\Stirling{n}{k}=\sum_Z G_Z,$$
where the summation is taken over all $k$-Young tableaux of shape $(k, (1^{n-k}0^{k-1}))$.
\end{proposition}
\begin{proof}
Let $c=x$ and $f=1/{(1-x)}$. Then $c_1=1$ and $c_j=0$ for $j\geq2$, and
$\mathbf{f}_k={k!}/(1-x)^{k+1}$.
It follows from~\eqref{eq:cD:p} that
	\begin{align*}
	\left(xD\right)^n \frac{1}{1-x}
			&= \sum_{(k, \mu)\vdash n} p_{k, \mu} c c_{\mu_1}c_{\mu_2}\cdots c_{\mu_{n-1}}  \mathbf{f}_k\\
			&= \sum_{(k, \mu=(1^{n-k}0^{k-1}))\vdash n} p_{k, \mu}\cdot
			\frac{k! x^k}{(1-x)^{k+1}}\\
			&= \frac{1}{(1-x)^{n+1}} \sum_{(k, \mu=(1^{n-k}0^{k-1}))\vdash n} p_{k, \mu}
			\cdot k!{x^{k}} (1-x)^{n-k}.
	\end{align*}
	By Theorem \ref{th:kmu:tableau},
	we have
	\begin{align}	
		A_n(x)
		&= \sum_{k=0}^n p_{k, (1^{n-k}0^{k-1})}\cdot k! x^k(1-x)^{n-k}\nonumber\\
		&= \sum_{k=0}^n \sum_Z G_Z\cdot k! x^k(1-x)^{n-k}\label{eq:A:GZ},
	\end{align}
where the second summation is taken over all $k$-Young tableaux of shape $(k, (1^{n-k}0^{k-1}))$. Recall that the Frobenius formula for Eulerian polynomials is given as follows (see~\cite{Chow08} for instance):
$$A_n(x)=\sum_{k=0}^nk!\Stirling{n}{k}x^k(1-x)^{n-k}.$$
By comparing with~\eqref{eq:A:GZ}, we get the desired result.
\end{proof}
\subsection{The $g$-index of Young tableaux}
\hspace*{\parindent}

Let $T$ be a standard Young tableau
of shape $\lambda$. We always put a special column of $n$ boxes at the left of $T$, and labelled by $1,2,3,\ldots, n$ from bottom to top.
For each $v\in [n]$, suppose that $v$ is in the box $(i,j)$,
we define the {\it $g$-index} of $v$, denoted by $g_{T}(v)$, to be the number of boxes $(i-1,j')$
such that $j'\geq j$ and the letter in this box is less than or equal to $v$ (see Figure \ref{fig:Ytab}, right diagram).
The {\it $g$-index} of $T$ is defined by
$$G_T=g_T(1) g_T(2) \cdots g_T(n).$$
For the Young tableau given in Figure \ref{fig:Ytab} (left diagram), we have
$$
g_T(1)=1, \
g_T(2)=1, \
g_T(3)=2, \
g_T(4)=1, \
g_T(5)=1, \
g_T(6)=4, \
g_T(7)=1.
$$
\begin{figure}[tbp]
\begin{center}
\begin{tikzpicture} 
\fill [gray!100](2.4000,5.6000)--(2.4000,6.4000)--(3.2000,6.4000)--(3.2000,5.6000)--(2.4000,5.6000);
\draw [gray!200](2.4000,5.6000)--(2.4000,6.4000)--(3.2000,6.4000)--(3.2000,5.6000)--(2.4000,5.6000);
\draw [gray!0] (2.8000, 6.0000) node [] {7};
\fill [gray!100](2.4000,4.8000)--(2.4000,5.6000)--(3.2000,5.6000)--(3.2000,4.8000)--(2.4000,4.8000);
\draw [gray!200](2.4000,4.8000)--(2.4000,5.6000)--(3.2000,5.6000)--(3.2000,4.8000)--(2.4000,4.8000);
\draw [gray!0] (2.8000, 5.2000) node [] {6};
\fill [gray!100](2.4000,4.0000)--(2.4000,4.8000)--(3.2000,4.8000)--(3.2000,4.0000)--(2.4000,4.0000);
\draw [gray!200](2.4000,4.0000)--(2.4000,4.8000)--(3.2000,4.8000)--(3.2000,4.0000)--(2.4000,4.0000);
\draw [gray!0] (2.8000, 4.4000) node [] {5};
\fill [gray!100](2.4000,3.2000)--(2.4000,4.0000)--(3.2000,4.0000)--(3.2000,3.2000)--(2.4000,3.2000);
\draw [gray!200](2.4000,3.2000)--(2.4000,4.0000)--(3.2000,4.0000)--(3.2000,3.2000)--(2.4000,3.2000);
\draw [gray!0] (2.8000, 3.6000) node [] {4};
\fill [gray!100](2.4000,2.4000)--(2.4000,3.2000)--(3.2000,3.2000)--(3.2000,2.4000)--(2.4000,2.4000);
\draw [gray!200](2.4000,2.4000)--(2.4000,3.2000)--(3.2000,3.2000)--(3.2000,2.4000)--(2.4000,2.4000);
\draw [gray!0] (2.8000, 2.8000) node [] {3};
\fill [gray!100](2.4000,1.6000)--(2.4000,2.4000)--(3.2000,2.4000)--(3.2000,1.6000)--(2.4000,1.6000);
\draw [gray!200](2.4000,1.6000)--(2.4000,2.4000)--(3.2000,2.4000)--(3.2000,1.6000)--(2.4000,1.6000);
\draw [gray!0] (2.8000, 2.0000) node [] {2};
\fill [gray!100](2.4000,0.8000)--(2.4000,1.6000)--(3.2000,1.6000)--(3.2000,0.8000)--(2.4000,0.8000);
\draw [gray!200](2.4000,0.8000)--(2.4000,1.6000)--(3.2000,1.6000)--(3.2000,0.8000)--(2.4000,0.8000);
\draw [gray!0] (2.8000, 1.2000) node [] {1};
None
\fill [gray!0](3.2000,2.4000)--(3.2000,3.2000)--(4.0000,3.2000)--(4.0000,2.4000)--(3.2000,2.4000);
\draw [gray!200](3.2000,2.4000)--(3.2000,3.2000)--(4.0000,3.2000)--(4.0000,2.4000)--(3.2000,2.4000);
\draw [gray!200] (3.6000, 2.8000) node [] {6};
\fill [gray!0](3.2000,1.6000)--(3.2000,2.4000)--(4.0000,2.4000)--(4.0000,1.6000)--(3.2000,1.6000);
\draw [gray!200](3.2000,1.6000)--(3.2000,2.4000)--(4.0000,2.4000)--(4.0000,1.6000)--(3.2000,1.6000);
\draw [gray!200] (3.6000, 2.0000) node [] {2};
\fill [gray!0](4.0000,1.6000)--(4.0000,2.4000)--(4.8000,2.4000)--(4.8000,1.6000)--(4.0000,1.6000);
\draw [gray!200](4.0000,1.6000)--(4.0000,2.4000)--(4.8000,2.4000)--(4.8000,1.6000)--(4.0000,1.6000);
\draw [gray!200] (4.4000, 2.0000) node [] {5};
\fill [gray!0](3.2000,0.8000)--(3.2000,1.6000)--(4.0000,1.6000)--(4.0000,0.8000)--(3.2000,0.8000);
\draw [gray!200](3.2000,0.8000)--(3.2000,1.6000)--(4.0000,1.6000)--(4.0000,0.8000)--(3.2000,0.8000);
\draw [gray!200] (3.6000, 1.2000) node [] {1};
\fill [gray!0](4.0000,0.8000)--(4.0000,1.6000)--(4.8000,1.6000)--(4.8000,0.8000)--(4.0000,0.8000);
\draw [gray!200](4.0000,0.8000)--(4.0000,1.6000)--(4.8000,1.6000)--(4.8000,0.8000)--(4.0000,0.8000);
\draw [gray!200] (4.4000, 1.2000) node [] {3};
\fill [gray!0](4.8000,0.8000)--(4.8000,1.6000)--(5.6000,1.6000)--(5.6000,0.8000)--(4.8000,0.8000);
\draw [gray!200](4.8000,0.8000)--(4.8000,1.6000)--(5.6000,1.6000)--(5.6000,0.8000)--(4.8000,0.8000);
\draw [gray!200] (5.2000, 1.2000) node [] {4};
\fill [gray!0](5.6000,0.8000)--(5.6000,1.6000)--(6.4000,1.6000)--(6.4000,0.8000)--(5.6000,0.8000);
\draw [gray!200](5.6000,0.8000)--(5.6000,1.6000)--(6.4000,1.6000)--(6.4000,0.8000)--(5.6000,0.8000);
\draw [gray!200] (6.0000, 1.2000) node [] {7};
None
\fill [gray!0](8.8000,5.6000)--(8.8000,6.0000)--(9.2000,6.0000)--(9.2000,5.6000)--(8.8000,5.6000);
\draw [gray!200](8.8000,5.6000)--(8.8000,6.0000)--(9.2000,6.0000)--(9.2000,5.6000)--(8.8000,5.6000);
\draw [gray!200] (9.0000, 5.8000) node [] {};
\fill [gray!0](8.8000,5.2000)--(8.8000,5.6000)--(9.2000,5.6000)--(9.2000,5.2000)--(8.8000,5.2000);
\draw [gray!200](8.8000,5.2000)--(8.8000,5.6000)--(9.2000,5.6000)--(9.2000,5.2000)--(8.8000,5.2000);
\draw [gray!200] (9.0000, 5.4000) node [] {};
\fill [gray!0](8.8000,4.8000)--(8.8000,5.2000)--(9.2000,5.2000)--(9.2000,4.8000)--(8.8000,4.8000);
\draw [gray!200](8.8000,4.8000)--(8.8000,5.2000)--(9.2000,5.2000)--(9.2000,4.8000)--(8.8000,4.8000);
\draw [gray!200] (9.0000, 5.0000) node [] {};
\fill [gray!0](8.8000,4.4000)--(8.8000,4.8000)--(9.2000,4.8000)--(9.2000,4.4000)--(8.8000,4.4000);
\draw [gray!200](8.8000,4.4000)--(8.8000,4.8000)--(9.2000,4.8000)--(9.2000,4.4000)--(8.8000,4.4000);
\draw [gray!200] (9.0000, 4.6000) node [] {};
\fill [gray!0](8.8000,4.0000)--(8.8000,4.4000)--(9.2000,4.4000)--(9.2000,4.0000)--(8.8000,4.0000);
\draw [gray!200](8.8000,4.0000)--(8.8000,4.4000)--(9.2000,4.4000)--(9.2000,4.0000)--(8.8000,4.0000);
\draw [gray!200] (9.0000, 4.2000) node [] {};
\fill [gray!0](9.2000,4.0000)--(9.2000,4.4000)--(9.6000,4.4000)--(9.6000,4.0000)--(9.2000,4.0000);
\draw [gray!200](9.2000,4.0000)--(9.2000,4.4000)--(9.6000,4.4000)--(9.6000,4.0000)--(9.2000,4.0000);
\draw [gray!200] (9.4000, 4.2000) node [] {};
\fill [gray!0](9.6000,4.0000)--(9.6000,4.4000)--(10.0000,4.4000)--(10.0000,4.0000)--(9.6000,4.0000);
\draw [gray!200](9.6000,4.0000)--(9.6000,4.4000)--(10.0000,4.4000)--(10.0000,4.0000)--(9.6000,4.0000);
\draw [gray!200] (9.8000, 4.2000) node [] {};
\fill [gray!0](8.8000,3.6000)--(8.8000,4.0000)--(9.2000,4.0000)--(9.2000,3.6000)--(8.8000,3.6000);
\draw [gray!200](8.8000,3.6000)--(8.8000,4.0000)--(9.2000,4.0000)--(9.2000,3.6000)--(8.8000,3.6000);
\draw [gray!200] (9.0000, 3.8000) node [] {};
\fill [gray!0](9.2000,3.6000)--(9.2000,4.0000)--(9.6000,4.0000)--(9.6000,3.6000)--(9.2000,3.6000);
\draw [gray!200](9.2000,3.6000)--(9.2000,4.0000)--(9.6000,4.0000)--(9.6000,3.6000)--(9.2000,3.6000);
\draw [gray!200] (9.4000, 3.8000) node [] {};
\fill [gray!0](9.6000,3.6000)--(9.6000,4.0000)--(10.0000,4.0000)--(10.0000,3.6000)--(9.6000,3.6000);
\draw [gray!200](9.6000,3.6000)--(9.6000,4.0000)--(10.0000,4.0000)--(10.0000,3.6000)--(9.6000,3.6000);
\draw [gray!200] (9.8000, 3.8000) node [] {};
\fill [gray!0](8.8000,3.2000)--(8.8000,3.6000)--(9.2000,3.6000)--(9.2000,3.2000)--(8.8000,3.2000);
\draw [gray!200](8.8000,3.2000)--(8.8000,3.6000)--(9.2000,3.6000)--(9.2000,3.2000)--(8.8000,3.2000);
\draw [gray!200] (9.0000, 3.4000) node [] {};
\fill [gray!0](9.2000,3.2000)--(9.2000,3.6000)--(9.6000,3.6000)--(9.6000,3.2000)--(9.2000,3.2000);
\draw [gray!200](9.2000,3.2000)--(9.2000,3.6000)--(9.6000,3.6000)--(9.6000,3.2000)--(9.2000,3.2000);
\draw [gray!200] (9.4000, 3.4000) node [] {$y$};
\fill [gray!0](9.6000,3.2000)--(9.6000,3.6000)--(10.0000,3.6000)--(10.0000,3.2000)--(9.6000,3.2000);
\draw [gray!200](9.6000,3.2000)--(9.6000,3.6000)--(10.0000,3.6000)--(10.0000,3.2000)--(9.6000,3.2000);
\draw [gray!200] (9.8000, 3.4000) node [] {};
\fill [gray!0](8.8000,2.8000)--(8.8000,3.2000)--(9.2000,3.2000)--(9.2000,2.8000)--(8.8000,2.8000);
\draw [gray!200](8.8000,2.8000)--(8.8000,3.2000)--(9.2000,3.2000)--(9.2000,2.8000)--(8.8000,2.8000);
\draw [gray!200] (9.0000, 3.0000) node [] {};
\fill [gray!0](9.2000,2.8000)--(9.2000,3.2000)--(9.6000,3.2000)--(9.6000,2.8000)--(9.2000,2.8000);
\draw [gray!200](9.2000,2.8000)--(9.2000,3.2000)--(9.6000,3.2000)--(9.6000,2.8000)--(9.2000,2.8000);
\draw [gray!200] (9.4000, 3.0000) node [] {};
\fill [gray!0](9.6000,2.8000)--(9.6000,3.2000)--(10.0000,3.2000)--(10.0000,2.8000)--(9.6000,2.8000);
\draw [gray!200](9.6000,2.8000)--(9.6000,3.2000)--(10.0000,3.2000)--(10.0000,2.8000)--(9.6000,2.8000);
\draw [gray!200] (9.8000, 3.0000) node [] {};
\fill [gray!0](8.8000,2.4000)--(8.8000,2.8000)--(9.2000,2.8000)--(9.2000,2.4000)--(8.8000,2.4000);
\draw [gray!200](8.8000,2.4000)--(8.8000,2.8000)--(9.2000,2.8000)--(9.2000,2.4000)--(8.8000,2.4000);
\draw [gray!200] (9.0000, 2.6000) node [] {};
\fill [gray!0](9.2000,2.4000)--(9.2000,2.8000)--(9.6000,2.8000)--(9.6000,2.4000)--(9.2000,2.4000);
\draw [gray!200](9.2000,2.4000)--(9.2000,2.8000)--(9.6000,2.8000)--(9.6000,2.4000)--(9.2000,2.4000);
\draw [gray!200] (9.4000, 2.6000) node [] {};
\fill [gray!0](9.6000,2.4000)--(9.6000,2.8000)--(10.0000,2.8000)--(10.0000,2.4000)--(9.6000,2.4000);
\draw [gray!200](9.6000,2.4000)--(9.6000,2.8000)--(10.0000,2.8000)--(10.0000,2.4000)--(9.6000,2.4000);
\draw [gray!200] (9.8000, 2.6000) node [] {};
\fill [gray!0](10.0000,2.4000)--(10.0000,2.8000)--(10.4000,2.8000)--(10.4000,2.4000)--(10.0000,2.4000);
\draw [gray!200](10.0000,2.4000)--(10.0000,2.8000)--(10.4000,2.8000)--(10.4000,2.4000)--(10.0000,2.4000);
\draw [gray!200] (10.2000, 2.6000) node [] {};
\fill [gray!0](8.8000,2.0000)--(8.8000,2.4000)--(9.2000,2.4000)--(9.2000,2.0000)--(8.8000,2.0000);
\draw [gray!200](8.8000,2.0000)--(8.8000,2.4000)--(9.2000,2.4000)--(9.2000,2.0000)--(8.8000,2.0000);
\draw [gray!200] (9.0000, 2.2000) node [] {};
\fill [gray!0](9.2000,2.0000)--(9.2000,2.4000)--(9.6000,2.4000)--(9.6000,2.0000)--(9.2000,2.0000);
\draw [gray!200](9.2000,2.0000)--(9.2000,2.4000)--(9.6000,2.4000)--(9.6000,2.0000)--(9.2000,2.0000);
\draw [gray!200] (9.4000, 2.2000) node [] {};
\fill [gray!0](9.6000,2.0000)--(9.6000,2.4000)--(10.0000,2.4000)--(10.0000,2.0000)--(9.6000,2.0000);
\draw [gray!200](9.6000,2.0000)--(9.6000,2.4000)--(10.0000,2.4000)--(10.0000,2.0000)--(9.6000,2.0000);
\draw [gray!200] (9.8000, 2.2000) node [] {};
\fill [gray!0](10.0000,2.0000)--(10.0000,2.4000)--(10.4000,2.4000)--(10.4000,2.0000)--(10.0000,2.0000);
\draw [gray!200](10.0000,2.0000)--(10.0000,2.4000)--(10.4000,2.4000)--(10.4000,2.0000)--(10.0000,2.0000);
\draw [gray!200] (10.2000, 2.2000) node [] {};
\fill [gray!0](8.8000,1.6000)--(8.8000,2.0000)--(9.2000,2.0000)--(9.2000,1.6000)--(8.8000,1.6000);
\draw [gray!200](8.8000,1.6000)--(8.8000,2.0000)--(9.2000,2.0000)--(9.2000,1.6000)--(8.8000,1.6000);
\draw [gray!200] (9.0000, 1.8000) node [] {};
\fill [gray!0](9.2000,1.6000)--(9.2000,2.0000)--(9.6000,2.0000)--(9.6000,1.6000)--(9.2000,1.6000);
\draw [gray!200](9.2000,1.6000)--(9.2000,2.0000)--(9.6000,2.0000)--(9.6000,1.6000)--(9.2000,1.6000);
\draw [gray!200] (9.4000, 1.8000) node [] {};
\fill [gray!0](9.6000,1.6000)--(9.6000,2.0000)--(10.0000,2.0000)--(10.0000,1.6000)--(9.6000,1.6000);
\draw [gray!200](9.6000,1.6000)--(9.6000,2.0000)--(10.0000,2.0000)--(10.0000,1.6000)--(9.6000,1.6000);
\draw [gray!200] (9.8000, 1.8000) node [] {$v$};
\fill [gray!0](10.0000,1.6000)--(10.0000,2.0000)--(10.4000,2.0000)--(10.4000,1.6000)--(10.0000,1.6000);
\draw [gray!200](10.0000,1.6000)--(10.0000,2.0000)--(10.4000,2.0000)--(10.4000,1.6000)--(10.0000,1.6000);
\draw [gray!200] (10.2000, 1.8000) node [] {};
\fill [gray!0](10.4000,1.6000)--(10.4000,2.0000)--(10.8000,2.0000)--(10.8000,1.6000)--(10.4000,1.6000);
\draw [gray!200](10.4000,1.6000)--(10.4000,2.0000)--(10.8000,2.0000)--(10.8000,1.6000)--(10.4000,1.6000);
\draw [gray!200] (10.6000, 1.8000) node [] {};
\fill [gray!0](8.8000,1.2000)--(8.8000,1.6000)--(9.2000,1.6000)--(9.2000,1.2000)--(8.8000,1.2000);
\draw [gray!200](8.8000,1.2000)--(8.8000,1.6000)--(9.2000,1.6000)--(9.2000,1.2000)--(8.8000,1.2000);
\draw [gray!200] (9.0000, 1.4000) node [] {};
\fill [gray!0](9.2000,1.2000)--(9.2000,1.6000)--(9.6000,1.6000)--(9.6000,1.2000)--(9.2000,1.2000);
\draw [gray!200](9.2000,1.2000)--(9.2000,1.6000)--(9.6000,1.6000)--(9.6000,1.2000)--(9.2000,1.2000);
\draw [gray!200] (9.4000, 1.4000) node [] {};
\fill [gray!0](9.6000,1.2000)--(9.6000,1.6000)--(10.0000,1.6000)--(10.0000,1.2000)--(9.6000,1.2000);
\draw [gray!200](9.6000,1.2000)--(9.6000,1.6000)--(10.0000,1.6000)--(10.0000,1.2000)--(9.6000,1.2000);
\draw [gray!200] (9.8000, 1.4000) node [] {};
\fill [gray!0](10.0000,1.2000)--(10.0000,1.6000)--(10.4000,1.6000)--(10.4000,1.2000)--(10.0000,1.2000);
\draw [gray!200](10.0000,1.2000)--(10.0000,1.6000)--(10.4000,1.6000)--(10.4000,1.2000)--(10.0000,1.2000);
\draw [gray!200] (10.2000, 1.4000) node [] {};
\fill [gray!0](10.4000,1.2000)--(10.4000,1.6000)--(10.8000,1.6000)--(10.8000,1.2000)--(10.4000,1.2000);
\draw [gray!200](10.4000,1.2000)--(10.4000,1.6000)--(10.8000,1.6000)--(10.8000,1.2000)--(10.4000,1.2000);
\draw [gray!200] (10.6000, 1.4000) node [] {};
\fill [gray!0](8.8000,0.8000)--(8.8000,1.2000)--(9.2000,1.2000)--(9.2000,0.8000)--(8.8000,0.8000);
\draw [gray!200](8.8000,0.8000)--(8.8000,1.2000)--(9.2000,1.2000)--(9.2000,0.8000)--(8.8000,0.8000);
\draw [gray!200] (9.0000, 1.0000) node [] { };
\fill [gray!0](9.2000,0.8000)--(9.2000,1.2000)--(9.6000,1.2000)--(9.6000,0.8000)--(9.2000,0.8000);
\draw [gray!200](9.2000,0.8000)--(9.2000,1.2000)--(9.6000,1.2000)--(9.6000,0.8000)--(9.2000,0.8000);
\draw [gray!200] (9.4000, 1.0000) node [] {};
\fill [gray!0](9.6000,0.8000)--(9.6000,1.2000)--(10.0000,1.2000)--(10.0000,0.8000)--(9.6000,0.8000);
\draw [gray!200](9.6000,0.8000)--(9.6000,1.2000)--(10.0000,1.2000)--(10.0000,0.8000)--(9.6000,0.8000);
\draw [gray!200] (9.8000, 1.0000) node [] {};
\fill [gray!0](10.0000,0.8000)--(10.0000,1.2000)--(10.4000,1.2000)--(10.4000,0.8000)--(10.0000,0.8000);
\draw [gray!200](10.0000,0.8000)--(10.0000,1.2000)--(10.4000,1.2000)--(10.4000,0.8000)--(10.0000,0.8000);
\draw [gray!200] (10.2000, 1.0000) node [] {};
\fill [gray!0](10.4000,0.8000)--(10.4000,1.2000)--(10.8000,1.2000)--(10.8000,0.8000)--(10.4000,0.8000);
\draw [gray!200](10.4000,0.8000)--(10.4000,1.2000)--(10.8000,1.2000)--(10.8000,0.8000)--(10.4000,0.8000);
\draw [gray!200] (10.6000, 1.0000) node [] {};
\fill [gray!0](10.8000,0.8000)--(10.8000,1.2000)--(11.2000,1.2000)--(11.2000,0.8000)--(10.8000,0.8000);
\draw [gray!200](10.8000,0.8000)--(10.8000,1.2000)--(11.2000,1.2000)--(11.2000,0.8000)--(10.8000,0.8000);
\draw [gray!200] (11.0000, 1.0000) node [] {};
None
\fill [gray!50](9.2000,2.8000)--(9.2000,3.2000)--(9.6000,3.2000)--(9.6000,2.8000)--(9.2000,2.8000);
\draw [gray!200](9.2000,2.8000)--(9.2000,3.2000)--(9.6000,3.2000)--(9.6000,2.8000)--(9.2000,2.8000);
\draw [gray!200] (9.4000, 3.0000) node [] {$x$};
\fill [gray!50](9.2000,2.4000)--(9.2000,2.8000)--(9.6000,2.8000)--(9.6000,2.4000)--(9.2000,2.4000);
\draw [gray!200](9.2000,2.4000)--(9.2000,2.8000)--(9.6000,2.8000)--(9.6000,2.4000)--(9.2000,2.4000);
\draw [gray!200] (9.4000, 2.6000) node [] {$x$};
\fill [gray!50](9.2000,2.0000)--(9.2000,2.4000)--(9.6000,2.4000)--(9.6000,2.0000)--(9.2000,2.0000);
\draw [gray!200](9.2000,2.0000)--(9.2000,2.4000)--(9.6000,2.4000)--(9.6000,2.0000)--(9.2000,2.0000);
\draw [gray!200] (9.4000, 2.2000) node [] {$x$};
\fill [gray!50](9.2000,1.6000)--(9.2000,2.0000)--(9.6000,2.0000)--(9.6000,1.6000)--(9.2000,1.6000);
\draw [gray!200](9.2000,1.6000)--(9.2000,2.0000)--(9.6000,2.0000)--(9.6000,1.6000)--(9.2000,1.6000);
\draw [gray!200] (9.4000, 1.8000) node [] {$x$};
None
\fill [gray!0](12.0000,4.0000)--(12.0000,4.4000)--(12.4000,4.4000)--(12.4000,4.0000)--(12.0000,4.0000);
\draw [gray!0](12.0000,4.0000)--(12.0000,4.4000)--(12.4000,4.4000)--(12.4000,4.0000)--(12.0000,4.0000);
\draw [gray!200] (12.2000, 4.2000) node [] {$x\leq v<y$};
\fill [gray!0](12.0000,3.2000)--(12.0000,3.6000)--(12.4000,3.6000)--(12.4000,3.2000)--(12.0000,3.2000);
\draw [gray!0](12.0000,3.2000)--(12.0000,3.6000)--(12.4000,3.6000)--(12.4000,3.2000)--(12.0000,3.2000);
\draw [gray!200] (12.2000, 3.4000) node [] {$g_T(v)=\#x$};
\end{tikzpicture} 
\end{center}
\caption{Young tableaux and $g$-index}
\label{fig:Ytab}
\end{figure}
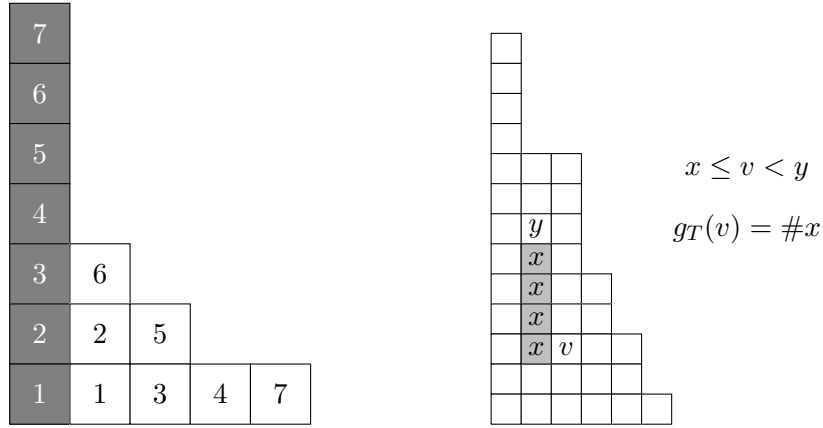

Let $\lambda(T)$ be the corresponding partition of the Young tableau $T$.
If $\lambda(T)=(\lambda_1,\lambda_2,\ldots,\lambda_{\ell})$, then let $\lambda(T)! =\lambda_1! \lambda_2!\cdots \lambda_\ell !$.

We now present the fourth main result of this paper.
\begin{theorem}\label{th:Cn}
Let $C_n(x)$ be the second-order Eulerian polynomials. Then we have
\begin{align}\label{eq:th:Cn}
	C_n(x) &= \sum_{T\in \SYT(n)} G_T\  \lambda(T) ! \ x^{n+1-\ell(\lambda(T))}.
\end{align}
\end{theorem}

Take $x=1$ in~\eqref{eq:th:Cn}, we obtain the following corollary.
\begin{corollary} We have
\begin{align*}
	(2n-1)!! &= \sum_{T\in \SYT(n)} G_T\  \lambda(T) !.
\end{align*}
\end{corollary}

\begin{example}
	For $n=4$, the 10 standard Young tableaux and their $g$-indexes
	are listed in Figure~\ref{fig:Ytab:example}. We verify that	
	$C_4(x)=24x^4 + 58 x^3 + 22 x^2 +x$.
\end{example}

\begin{figure}[tbp]
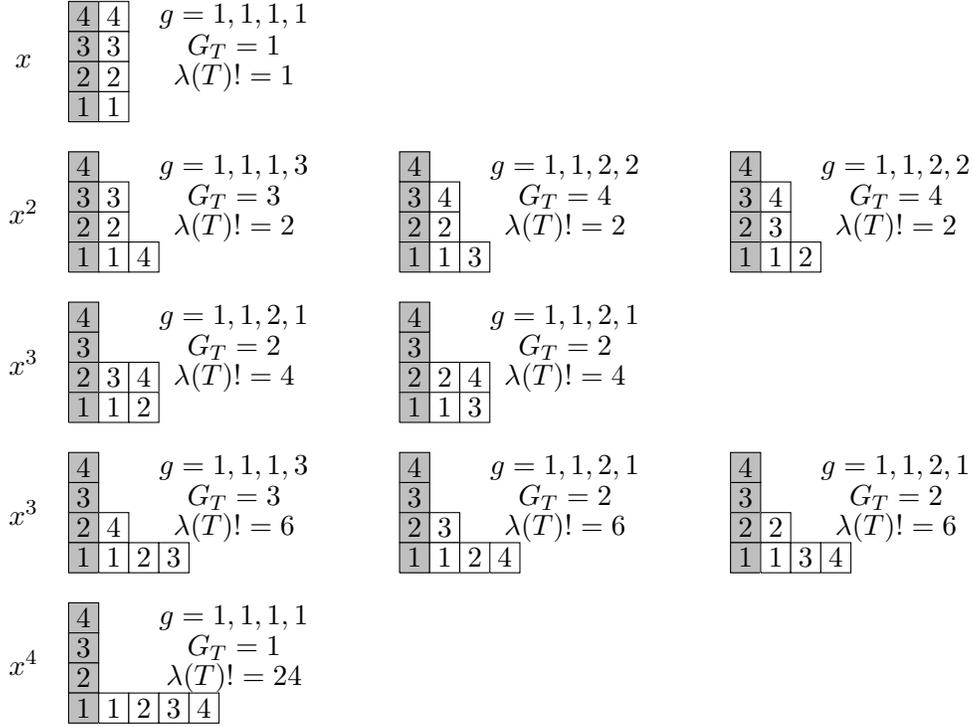

\begin{center}
\end{center}
\caption{$\SYT(4)$ and $g$-indexes}
\label{fig:Ytab:example}
\end{figure}

We now present the fifth main result of this paper.
\begin{theorem}\label{th:An}
Let $A_n(x)$ be the Eulerian polynomials. Then we have
\begin{align}\label{eq:th:An}
	A_n(x)= \sum_{T\in \SYT(n)} G_T\ x^{n+1-\ell(\lambda(T))}.
\end{align}
\end{theorem}

So the following corollary is immediate.
\begin{corollary}
We have
\begin{align*}
	n! &= \sum_{T\in \SYT(n)} G_T.
\end{align*}
\end{corollary}

\begin{example}
	For $n=4$, the 10 standard Young tableaux and their $g$-indexes
	are listed in Figure \ref{fig:Ytab:example}. We verify that	
	$A_4(x)=x^4+11x^3+11x^2+x$.
\end{example}

Let $\pi\in\msn$.
We say that $\pi$ has no {\it double descents} if there is no
index $i\in [n-2]$ such that $\pi(i)>\pi(i+1)>\pi(i+2)$.
The permutation $\pi$ is called {\it simsun} if for each $k\in [n]$, the
subword of $\pi$ restricted to $[k]$ (in the order
they appear in $\pi$) contains no double descents.
Simsun permutations are useful in describing the action of the symmetric group on the
maximal chains of the partition lattice (see~\cite{Sundaram1994,Sundaram1995}).
There has been much recent work devoted to simsun permutation and its variations, see~\cite{Chow11,Ma16} and references therein.

Denote by $\rs_n$ the set of simsun permutations in $\msn$.
Let $$S_n(x)=\sum_{\pi\in\rs_n}x^{\des(\pi)}=\sum_{i=1}^{\lrf{(n+2)/2}}S(n,i)x^i$$
be the descent polynomial of simsum permutations.
It follows from~\cite[Theorem~1]{Chow11} that the polynomials $S_n(x)$ satisfy the recurrence relation
\begin{equation*}\label{Snx-recu}
S_n(x)=(n+1)xS_{n-1}(x)+x(1-2x)S_{n-1}'(x)
\end{equation*}
for $n\geq 2$, with the initial conditions $S_{1}(x)=x,S_2(x)=x+x^2$ and $S_3(x)=x+4x^2$.

An {\it increasing tree} on $[n]$ is a rooted tree with vertex set $\{0,1,2,\ldots,n\}$ in which the
labels of the vertices are increasing along any path from the root $0$.
The {\it degree} of a vertex in a rooted tree is the number of its children.
A {\it 0-1-2 increasing tree} is an increasing tree in which the degree of any vertex is
at most two. It should be noted that the number $S(n,i)$ counts 0-1-2 increasing trees on $[n]$ with $i$ leaves (see~\cite[A094503]{Sloane}), and
the polynomial $S_n(x)$ is also known as the
{\it Andr\'e polynomial} (see~\cite{Chen17,Foata01}).

Now we present the sixth main result of this paper.
\begin{theorem}\label{thm:Sn}
Let $S_n(x)$ be the Andr\'e polynomials.
For $n\geq 1$, we have
\begin{equation}\label{SnxSYT}
	S_n(x)= \sum_{T} G_T\ x^{n+1-\ell(\lambda(T))},
\end{equation}
where the summation is taken
over all Young tableaux in $\SYT(n)$ with at most two columns.
\end{theorem}

We say that $\pi\in\msn$ is {\it alternating} if $\pi(1)>\pi(2)<\pi(3)>\cdots \pi(n)$. In other words, $\pi(i)<\pi({i+1})$ if $i$ is even and $\pi(i)>\pi({i+1})$ if $i$ is odd.
The {\it Euler number} $E_n$ is the number of alternating permutations in $\msn$ (see~\cite{Stanley}).
A remarkable property of simsun permutations is that $\#\rs_n=E_{n+1}$ (see~\cite[p.~267]{Sundaram1994}).
So we get the following corollary.
\begin{corollary}
Let $E_n$ be the $n$th Euler number. Then we have
\begin{align*}
	E_{n+1} &= \sum_{T} G_T,
\end{align*}
where the summation is taken
over all Young tableaux in $\SYT(n)$ with at most two columns.
\end{corollary}

An index $i\in [n]$ is a {\it peak} (resp.~{\it exterior double descent})
of $\pi$ if $\pi(i-1)<\pi(i)>\pi(i+1)$ (resp. $\pi(i-1)>\pi(i)>\pi(i+1)$), where $\pi(0)=\pi(n+1)=0$.
Let $a(n,i)$ be the number of permutations in $\msn$ with $i$ peaks and without exterior double descents.
The following gamma expansion of Eulerian polynomials was first given by Foata and Sch\"utzenberger~\cite{Foata70}:
\begin{equation*}\label{Anx-gamma}
A_n(x)=\sum_{i=1}^{\lrf{({n+1})/{2}}}a(n,i)x^i(1+x)^{n+1-2i},
\end{equation*}
which implies that Eulerian polynomials are symmetric and unimodal. In recent years there has been much interest in studying
gamma expansions of combinatorial polynomials, see~\cite{Lin15,Ma1902} and the references therein.
Combining~\cite[Corollary~3.2]{Branden08} and~\cite[Proposition~1]{Ma16}, we get another gamma expansion of Eulerian polynomials:
$$A_{n+1}(x)=\sum_{i=1}^{\lrf{(n+2)/2}}2^{i-1}S(n,i)x^i(1+x)^{n+2-2i}.$$

Let $T$ be a Young tableau in $\SYT(n)$ with at most two columns. If $\lambda(T)=(1^{n-2i+2}2^{i-1})$, then $n+1-\ell(\lambda(T))=i$, where
$1\leq i\leq \lrf{(n+2)/2}$. Then by using~\eqref{SnxSYT}, we immediately get the following result.
\begin{theorem}
Let $S(n,i)$ be the number of 0-1-2 increasing trees on $[n]$ with $i$ leaves.
Then
\begin{equation*}\label{RnxSYT}
\sum_{i=1}^{\lrf{(n+2)/2}}2^{i-1}S(n,i)x^i= \sum_{T} G_T\lambda(T)!\ x^{n+1-\ell(\lambda(T))},
\end{equation*}
where the summation is taken
over all Young tableaux in $\SYT(n)$ with at most two columns.
\end{theorem}
\section{Proof of Theorem~\ref{th:Cn}}\label{section04}
Setting $c=x/{(1-x)}$ and $f=1/{(1-x)}$, then we have
\begin{equation*}
		c_j=\frac{j!}{(1-x)^{j+1}} \quad (j\geq 1); \qquad
		\mathbf{f}_k=\frac{k!}{(1-x)^{k+1}} \quad  (k\geq 0).
\end{equation*}
By using~\eqref{eq:cD:p}, we obtain
	\begin{align*}
	\left(\frac{x}{1-x} D\right)^n \frac{1}{1-x}
			&= \sum_{(k, \mu)\vdash n} p_{k, \mu}\cdot c c_{\mu_1}c_{\mu_2}\cdots c_{\mu_{n-1}}  \mathbf{f}_k\\
			&= \sum_{(k, \mu)\vdash n} p_{k, \mu}\cdot
			\frac{x^{|\mu|_0+1}}{1-x} \frac{\mu_1!}{(1-x)^{\mu_1+1}}
			\cdots \frac{\mu_{n-1}!}{(1-x)^{\mu_{n-1}+1}}\frac{k!}{(1-x)^{k+1}}\\
			&= \frac{1}{(1-x)^{2n+1}} \sum_{(k, \mu)\vdash n} p_{k, \mu}\cdot
			k! {\mu_1!}
			\cdots {\mu_{n-1}!}{x^{|\mu|_0+1}},
	\end{align*}
	where the summation is taken over all types $(k,\mu)$ of $n$.
Combining~\eqref{eq:Cnx} and Theorem~\ref{th:kmu:tableau},
	we have
	\begin{align}	
		C_n(x)
		&= \sum_{(k, \mu)\vdash n} p_{k, \mu}\cdot k! {\mu_1!} \cdots {\mu_{n-1}!}{x^{|\mu|_0+1}}\nonumber\\
		&= \sum_{(k, \mu)\vdash n} \sum_Z G_Z\cdot k! {\mu_1!} \cdots {\mu_{n-1}!}{x^{|\mu|_0+1}}.\label{eq:C:GZ}
	\end{align}

In view of \eqref{eq:th:Cn} and \eqref{eq:C:GZ}, we need to establish some relations between $k$-Young tableaux and standard Young tableaux.
Let $Z$ be a $k$-Young tableau of shape $(k,\mu)$. We define $T=\rho(Z)$ to be the unique standard Young tableau
such that the sets of the letters in the $j$-th column in $Z$ and $T$ are the same for all $j$.
Let us list some basic facts of this map $Z\mapsto T=\rho(Z)$:
\begin{itemize}
  \item [$(i)$] We can obtain $T$ from $Z$ by ordering the letters in each column in increasing order. One can check that if $T$ is obtained in this way, then $T$ is a standard Young tableau;
  \item [$(ii)$] The partition $\lambda(T)$ is the
	decreasing ordering of the sequence $(k, \mu_1, \ldots, \mu_{n-1})$, removing the $0$'s at the end.
	Hence, $\lambda(T)!=k!\mu_1!\mu_2!\cdots \mu_{n-1}!$;
  \item [$(iii)$] We have $n-\ell(\lambda(T))= |\mu|_0$;
  \item [$(iv)$] In general $G_Z \not= G_T$.
\end{itemize}

For example, take the $k$-Young tableau given in Figure \ref{fig:Z:diag}, we obtain the standard Young tableau
	given in Figure \ref{fig:rhoZ}.
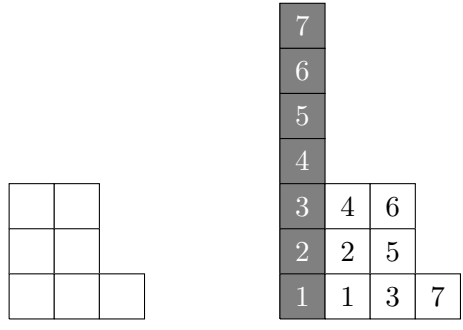
\begin{figure}[tbp]
\begin{center}
\begin{tikzpicture} 
\fill [gray!0](0.0000,1.8000)--(0.0000,2.4000)--(0.6000,2.4000)--(0.6000,1.8000)--(0.0000,1.8000);
\draw [gray!200](0.0000,1.8000)--(0.0000,2.4000)--(0.6000,2.4000)--(0.6000,1.8000)--(0.0000,1.8000);
\draw [gray!200] (0.3000, 2.1000) node [] {};
\fill [gray!0](0.6000,1.8000)--(0.6000,2.4000)--(1.2000,2.4000)--(1.2000,1.8000)--(0.6000,1.8000);
\draw [gray!200](0.6000,1.8000)--(0.6000,2.4000)--(1.2000,2.4000)--(1.2000,1.8000)--(0.6000,1.8000);
\draw [gray!200] (0.9000, 2.1000) node [] {};
\fill [gray!0](0.0000,1.2000)--(0.0000,1.8000)--(0.6000,1.8000)--(0.6000,1.2000)--(0.0000,1.2000);
\draw [gray!200](0.0000,1.2000)--(0.0000,1.8000)--(0.6000,1.8000)--(0.6000,1.2000)--(0.0000,1.2000);
\draw [gray!200] (0.3000, 1.5000) node [] {};
\fill [gray!0](0.6000,1.2000)--(0.6000,1.8000)--(1.2000,1.8000)--(1.2000,1.2000)--(0.6000,1.2000);
\draw [gray!200](0.6000,1.2000)--(0.6000,1.8000)--(1.2000,1.8000)--(1.2000,1.2000)--(0.6000,1.2000);
\draw [gray!200] (0.9000, 1.5000) node [] {};
\fill [gray!0](0.0000,0.6000)--(0.0000,1.2000)--(0.6000,1.2000)--(0.6000,0.6000)--(0.0000,0.6000);
\draw [gray!200](0.0000,0.6000)--(0.0000,1.2000)--(0.6000,1.2000)--(0.6000,0.6000)--(0.0000,0.6000);
\draw [gray!200] (0.3000, 0.9000) node [] {};
\fill [gray!0](0.6000,0.6000)--(0.6000,1.2000)--(1.2000,1.2000)--(1.2000,0.6000)--(0.6000,0.6000);
\draw [gray!200](0.6000,0.6000)--(0.6000,1.2000)--(1.2000,1.2000)--(1.2000,0.6000)--(0.6000,0.6000);
\draw [gray!200] (0.9000, 0.9000) node [] {};
\fill [gray!0](1.2000,0.6000)--(1.2000,1.2000)--(1.8000,1.2000)--(1.8000,0.6000)--(1.2000,0.6000);
\draw [gray!200](1.2000,0.6000)--(1.2000,1.2000)--(1.8000,1.2000)--(1.8000,0.6000)--(1.2000,0.6000);
\draw [gray!200] (1.5000, 0.9000) node [] {};
None
\fill [gray!100](3.6000,4.2000)--(3.6000,4.8000)--(4.2000,4.8000)--(4.2000,4.2000)--(3.6000,4.2000);
\draw [gray!200](3.6000,4.2000)--(3.6000,4.8000)--(4.2000,4.8000)--(4.2000,4.2000)--(3.6000,4.2000);
\draw [gray!0] (3.9000, 4.5000) node [] {7};
\fill [gray!100](3.6000,3.6000)--(3.6000,4.2000)--(4.2000,4.2000)--(4.2000,3.6000)--(3.6000,3.6000);
\draw [gray!200](3.6000,3.6000)--(3.6000,4.2000)--(4.2000,4.2000)--(4.2000,3.6000)--(3.6000,3.6000);
\draw [gray!0] (3.9000, 3.9000) node [] {6};
\fill [gray!100](3.6000,3.0000)--(3.6000,3.6000)--(4.2000,3.6000)--(4.2000,3.0000)--(3.6000,3.0000);
\draw [gray!200](3.6000,3.0000)--(3.6000,3.6000)--(4.2000,3.6000)--(4.2000,3.0000)--(3.6000,3.0000);
\draw [gray!0] (3.9000, 3.3000) node [] {5};
\fill [gray!100](3.6000,2.4000)--(3.6000,3.0000)--(4.2000,3.0000)--(4.2000,2.4000)--(3.6000,2.4000);
\draw [gray!200](3.6000,2.4000)--(3.6000,3.0000)--(4.2000,3.0000)--(4.2000,2.4000)--(3.6000,2.4000);
\draw [gray!0] (3.9000, 2.7000) node [] {4};
\fill [gray!100](3.6000,1.8000)--(3.6000,2.4000)--(4.2000,2.4000)--(4.2000,1.8000)--(3.6000,1.8000);
\draw [gray!200](3.6000,1.8000)--(3.6000,2.4000)--(4.2000,2.4000)--(4.2000,1.8000)--(3.6000,1.8000);
\draw [gray!0] (3.9000, 2.1000) node [] {3};
\fill [gray!100](3.6000,1.2000)--(3.6000,1.8000)--(4.2000,1.8000)--(4.2000,1.2000)--(3.6000,1.2000);
\draw [gray!200](3.6000,1.2000)--(3.6000,1.8000)--(4.2000,1.8000)--(4.2000,1.2000)--(3.6000,1.2000);
\draw [gray!0] (3.9000, 1.5000) node [] {2};
\fill [gray!100](3.6000,0.6000)--(3.6000,1.2000)--(4.2000,1.2000)--(4.2000,0.6000)--(3.6000,0.6000);
\draw [gray!200](3.6000,0.6000)--(3.6000,1.2000)--(4.2000,1.2000)--(4.2000,0.6000)--(3.6000,0.6000);
\draw [gray!0] (3.9000, 0.9000) node [] {1};
None
\fill [gray!0](4.2000,1.8000)--(4.2000,2.4000)--(4.8000,2.4000)--(4.8000,1.8000)--(4.2000,1.8000);
\draw [gray!200](4.2000,1.8000)--(4.2000,2.4000)--(4.8000,2.4000)--(4.8000,1.8000)--(4.2000,1.8000);
\draw [gray!200] (4.5000, 2.1000) node [] {4};
\fill [gray!0](4.8000,1.8000)--(4.8000,2.4000)--(5.4000,2.4000)--(5.4000,1.8000)--(4.8000,1.8000);
\draw [gray!200](4.8000,1.8000)--(4.8000,2.4000)--(5.4000,2.4000)--(5.4000,1.8000)--(4.8000,1.8000);
\draw [gray!200] (5.1000, 2.1000) node [] {6};
\fill [gray!0](4.2000,1.2000)--(4.2000,1.8000)--(4.8000,1.8000)--(4.8000,1.2000)--(4.2000,1.2000);
\draw [gray!200](4.2000,1.2000)--(4.2000,1.8000)--(4.8000,1.8000)--(4.8000,1.2000)--(4.2000,1.2000);
\draw [gray!200] (4.5000, 1.5000) node [] {2};
\fill [gray!0](4.8000,1.2000)--(4.8000,1.8000)--(5.4000,1.8000)--(5.4000,1.2000)--(4.8000,1.2000);
\draw [gray!200](4.8000,1.2000)--(4.8000,1.8000)--(5.4000,1.8000)--(5.4000,1.2000)--(4.8000,1.2000);
\draw [gray!200] (5.1000, 1.5000) node [] {5};
\fill [gray!0](4.2000,0.6000)--(4.2000,1.2000)--(4.8000,1.2000)--(4.8000,0.6000)--(4.2000,0.6000);
\draw [gray!200](4.2000,0.6000)--(4.2000,1.2000)--(4.8000,1.2000)--(4.8000,0.6000)--(4.2000,0.6000);
\draw [gray!200] (4.5000, 0.9000) node [] {1};
\fill [gray!0](4.8000,0.6000)--(4.8000,1.2000)--(5.4000,1.2000)--(5.4000,0.6000)--(4.8000,0.6000);
\draw [gray!200](4.8000,0.6000)--(4.8000,1.2000)--(5.4000,1.2000)--(5.4000,0.6000)--(4.8000,0.6000);
\draw [gray!200] (5.1000, 0.9000) node [] {3};
\fill [gray!0](5.4000,0.6000)--(5.4000,1.2000)--(6.0000,1.2000)--(6.0000,0.6000)--(5.4000,0.6000);
\draw [gray!200](5.4000,0.6000)--(5.4000,1.2000)--(6.0000,1.2000)--(6.0000,0.6000)--(5.4000,0.6000);
\draw [gray!200] (5.7000, 0.9000) node [] {7};
None
\end{tikzpicture} 
\end{center}
\caption{$T=\rho(Z)$ for $Z$ given in Figure \ref{fig:Z:diag}}
\label{fig:rhoZ}
\end{figure}
However the map $\rho$ is not bijective. Let $$\rho^{-1}(T) = \{(k,\mu, Z) \mid \rho(Z)=T\}.$$ By the above properties of $\rho$ and \eqref{eq:C:GZ}, we have
	\begin{align}	
		C_n(x)
		&= \sum_{T\in \SYT(n)}\sum_{(k, \mu, Z)\in \rho^{-1}(T)}  G_Z\cdot k! {\mu_1!} \cdots {\mu_{n-1}!}{x^{|\mu|_0+1}}\nonumber\\
		&= \sum_{T\in \SYT(n)} \lambda(T)! x^{n+1-\ell(\lambda(T))}\sum_{(k, \mu, Z)\in \rho^{-1}(T)}  G_Z.\label{eq:C:TZ}
	\end{align}

The following lemma is fundamental.
\begin{lemma}\label{lemmaGTZ}
For each standard Young tableau $T$, we have
\begin{equation}\label{eq:GTGZ}
	\sum_{Z\in \rho^{-1}(T)} G_Z = G_T,
\end{equation}
where we write	
${Z\in \rho^{-1}(T)} $
instead of
${(k,\mu, Z)\in \rho^{-1}(T)} $ since we can recover $(k, \mu)$ from $Z$.
\end{lemma}
\begin{proof}
We will proof \eqref{eq:GTGZ} by induction on the size of $T$.
Suppose that \eqref{eq:GTGZ} is true for all standard Young tableau $T$ of size $n-1$.
Given a $T\in\SYT(n)$. Let $T'$ is a standard Young tableau of size $n-1$ obtained from $T$ by removing the letter $n$. This operation is reversible if $\lambda(T)$ is known.
By the hypothesis of induction, we have
\begin{equation}\label{eq:GTGZ'}
	\sum_{Z'\in \rho^{-1}(T')} G_{Z'} = G_{T'},
\end{equation}
It should be noted that $$G_T= G_{T'} \times g_T(n).$$
On the other hand, for a $k$-Young tableau $Z\in \rho^{-1}(T)$ of size $n$,
if we remove the letter $n$, we obtain a $k'$-Young tableau $Z'\in\rho^{-1}(T')$ of sie $n-1$. However, unlike Young tableau, this operation is not always reversible. Let us analyse in detail.
Let $\beta$ be the length of the row containing the letter $n$ in $k$-Young tableau $Z\in\rho^{-1}(T)$ with shape $(k,\mu)$ if $n$ is in the top Young tableau of $Z$.
The set $\rho^{-1}(T)$ can be divided into four subsets:
$\rho^{-1}(T)=\Gamma_1 + \Gamma_2 + \Gamma_3 + \Gamma_4$,
where $\Gamma_1,\Gamma_2,\Gamma_3$ and $\Gamma_4$ are respectively defined as follows:
\begin{align*}
\Gamma_1& =\{Z\in\rho^{-1}(T): {\text{$n$ is in the top Young tableau and $k=\beta-1$}}\}, \\
\Gamma_2& =\{Z\in\rho^{-1}(T): {\text{$n$ is in the bottom row and $k-1\in \mu$}}\}, \\
\Gamma_3& =\{Z\in\rho^{-1}(T): {\text{$n$ is in the top Young tableau and $k\not=\beta-1$}}\},\\
\Gamma_4& =\{Z\in\rho^{-1}(T): {\text{$n$ is in the bottom row and $k-1\not\in \mu$}}\}.
\end{align*}
See Figure~\ref{fig:Gamma1234} for two examples.
It should be noted that some of the $\Gamma_i$ may be empty according to $T$.
\begin{figure}[tbp]
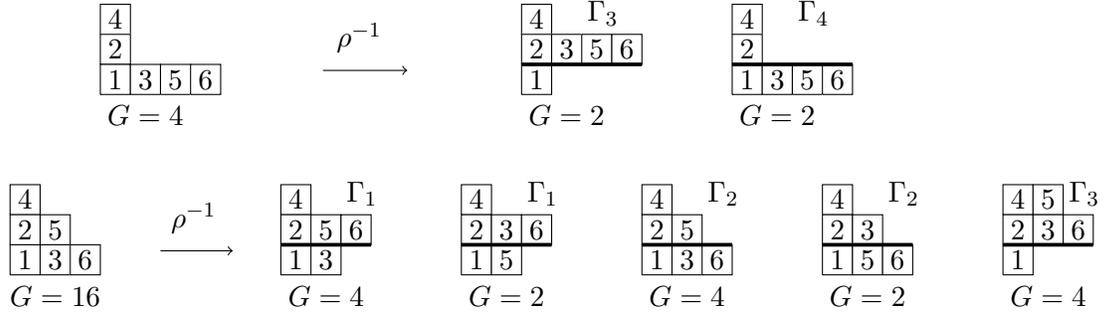

\begin{center}
\end{center}
\caption{Decomposition of $\rho^{-1}(T)$ into $\Gamma_1, \Gamma_2, \Gamma_3, \Gamma_4$}
\label{fig:Gamma1234}
\end{figure}
We claim that the set $\Gamma_1$ and $\Gamma_2$ have the same carnality.
Moreover, for each $Z_1\in \Gamma_1$, there exists
$Z_2\in\Gamma_2$ in a unique manner,
such that
$Z_1' = Z_2' \in \rho^{-1}(T')$,
See Figure~\ref{fig:twocases}.
\begin{figure}[tbp]
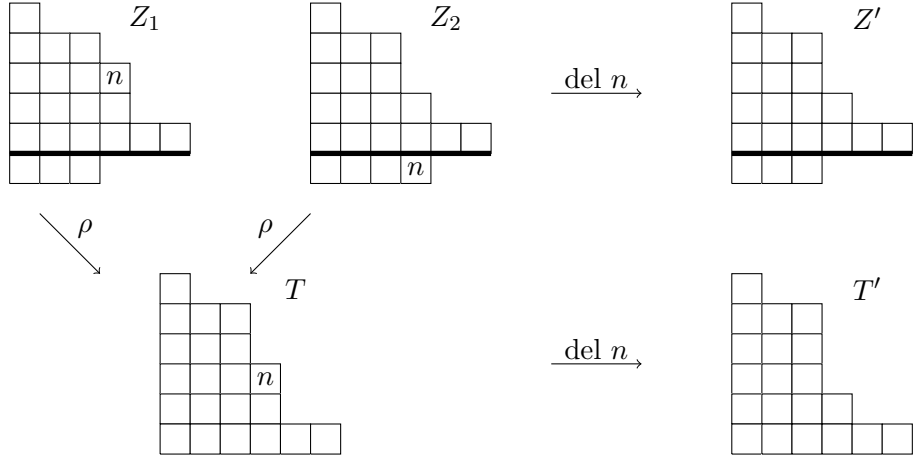

\begin{center}
%
\end{center}
\caption{$Z_1, Z_2\in \rho^{-1}(T)$ are mapped to the same $Z'\in\rho^{-1}(T')$ by removing the letter $n$}
\label{fig:twocases}
\end{figure}
Moreover, we have the relations for the $g$-indexes (see Figure~\ref{fig:twocases}):
$g_{Z_1}(n)=g_{T}(n)-1$ and $g_{Z_2}(n)=1$.
For $Z_3\in \Gamma_3$ and $Z_4\in \Gamma_4$ we have
$g_{Z_3}(n)=g_{T}(n)$ and $g_{Z_4}(n)=g_{T}(n)$.
By all these observations, we have
\begin{align*}
	\sum_{Z\in\rho^{-1}(T)} G_Z &=
	\sum_{Z_1\in \Gamma_1, Z_2\in\Gamma_2} (G_{Z_1} + G_{Z_2})
	+\sum_{Z_3\in \Gamma_3} G_{Z_3} +\sum_{Z_4\in \Gamma_4} G_{Z_4}
	\\
	&=	\sum_{Z_1\in \Gamma_1, Z_2\in\Gamma_2} (g_{Z_1}(n)G_{Z'} + g_{Z_2}(n)G_{Z'})
	+\sum_{Z_3\in \Gamma_3} g_T(n) G_{Z_3'}
	+\sum_{Z_4\in \Gamma_4} g_{T}(n)G_{Z_4'}
	\\
	&=g_T(n)	\sum_{Z'\in\rho^{-1}(T')}G_{Z'} \\
	&=g_T(n)	G_{T'}\\
	&= G_T.
\end{align*}
Hence~\eqref{eq:GTGZ} holds. This completes the proof.
\end{proof}

\begin{proof}[Proof of Theorem~\ref{th:Cn}]
Combining~\eqref{eq:C:TZ} and Lemma~\ref{lemmaGTZ}, we get that
	\begin{align*}	
		C_n(x)&= \sum_{T\in \SYT(n)} \lambda(T)! x^{n+1-\ell(\lambda(T))}\sum_{(k, \mu, Z)\in \rho^{-1}(T)}  G_Z\\
&=\sum_{T\in \SYT(n)} G_T\lambda(T)! x^{n+1-\ell(\lambda(T))},
	\end{align*}
as desired. This completes the proof.
\end{proof}
\section{Proof of Theorem~\ref{th:An}}\label{section05}
We shall prove Theorem~\ref{th:An} by using context-free grammars.
For an alphabet $V$, let $\mathbb{Q}[[V]]$ be the ring of the rational commutative ring of formal power series in
monomials formed from letters in $V$.
Following Chen~\cite{Chen93}, a {\it context-free grammar} over
$V$ is a function $G: V\rightarrow \mathbb{Q}[[V]]$ that replaces each letter in $V$ with an element of $\mathbb{Q}[[V]]$.
The formal derivative $D_G$ is a linear operator defined with respect to the grammar $G$.
In other words, $D_G$ is the unique derivation satisfying $D_G(x)=G(x)$ for $x\in V$.
For example, if $V=\{x,y\}$ and $G=\{x\rightarrow xy,y\rightarrow y\}$, then $D_G(x)=xy,D_G^2(x)=D_G(xy)=xy^2+xy$.
For two formal functions $u$ and $v$,
we have $D_G(u+v)=D_G(u)+D_G(v)$ and $D_G(uv)=D_G(u)v+uD_G(v)$.
For a constant $c$, we have $D_G(c)=0$.
It follows from {\it Leibniz's rule} that
$$D_G^n(uv)=\sum_{k=0}^n\binom{n}{k}D_G^k(u)D_G^{n-k}(v).$$
We refer the reader to~\cite{Chen17,Ma1902} for the recent progress on context-free grammars.

Setting $u_i=D_G^i(u)$, it follows from~\eqref{eq:cD:p} and~\eqref{eq:kmu:tableau} that
\begin{equation}\label{eq:uDG}
	\left(uD_G\right)^n= \sum_{(k, \mu)\vdash n}  \sum_{Z} G_Z u u_{\mu_1}u_{\mu_2}\cdots u_{\mu_{n-1}}D_G^k,
\end{equation}
where the first summation is taken over all types $(k, \mu)$ of $n$ and the second summation is taken over all $k$-Young tableaux of shape $(k,\mu)$.
It is well-known that Eulerian polynomials are symmetric, i.e., $A_0(x)=1$ and
$$A_n(x)=\sum_{i=1}^{n}\Eulerian{n}{i}x^{i} =\sum_{i=1}^{n}\Eulerian{n}{i}x^{n+1-i}~\text{for $n\geq1$}.$$
There is a grammatical interpretation of Eulerian numbers due to Dumont~\cite{Dumont96}, which can be restated as follows.
\begin{proposition}\label{Dumont2}
If
$G=\{x\rightarrow y, y\rightarrow y\}$,
then we have
$$(xD_G)^n(y)=\sum_{i=1}^{n}\Eulerian{n}{i}x^{n+1-i}y^{i}~\text{for $n\geq1$}.$$
\end{proposition}
\begin{proof}[Proof of Theorem~\ref{th:An}]
Let $G=\{x\rightarrow y, y\rightarrow y\}$. From~\eqref{eq:uDG}, we have
\begin{equation*}\label{eq:yDG}
	\left(xD_G\right)^n(y)= \sum_{(k, \mu)\vdash n}  \sum_{Z} G_Z x x_{\mu_1}x_{\mu_2}\cdots x_{\mu_{n-1}}D_G^k(y),
\end{equation*}
where $x_0=x$ and $x_i=D_G^i(x)=y$ for $i\geq 1$ and $D_G^k(y)=y$ for $k\geq 0$.
Hence
\begin{align*}
	\left(xD_G\right)^n(y)=\sum_{(k, \mu)\vdash n}  \sum_{Z} G_Z y^{n-|\mu|_0}x^{|\mu|_0+1}.
\end{align*}
Comparing this with Proposition~\ref{Dumont2}, we get
\begin{equation}
A_n(x) =\sum_{i=1}^{n}\Eulerian{n}{i}x^{n+1-i}
	=\left. (xD_G)^n (y) \right|_{y=1}
	=
	\sum_{(k, \mu)\vdash n}  \sum_{Z} G_Z x^{|\mu|_0+1},
\end{equation}
where the first summation is taken over all types $(k, \mu)$ of $n$ and the second summation is taken over all $k$-Young tableaux of shape $(k,\mu)$.
In the same way as the proof of Theorem~\ref{th:Cn}, by using Lemma~\ref{lemmaGTZ}, we get~\eqref{eq:th:An}.
\end{proof}
\section{Proof of Theorem~\ref{thm:Sn}}\label{section06}
We now recall a grammatical interpretation of $S_n(x)$.
\begin{proposition}[{\cite{Chen17,Dumont96}}]\label{grammar01:Snk}
Let $G_1=\{x\rightarrow xy, y\rightarrow x\}$. For $n\geq 1$, we have
$$D_{G_1}^n(x)=\sum_{i=1}^{\lrf{(n+2)/2}}S(n,i)x^iy^{n+2-2i}.$$
Thus $S_n(x)=D_{G_1}^n(x)|_{y=1}$.
\end{proposition}

It is routine to verify that Proposition~\ref{grammar01:Snk} can be restated as follows.
\begin{proposition}\label{grammar02:Snk}
Let $G_2=\{x\rightarrow y, y\rightarrow 1\}$. For $n\geq 1$, we have
$$(xD_{G_2})^n(x)=\sum_{i=1}^{\lrf{(n+2)/2}}S(n,i)x^iy^{n+2-2i}.$$
\end{proposition}

\begin{proof}[Proof of Theorem~\ref{thm:Sn}]
Let $G_2=\{x\rightarrow y, y\rightarrow 1\}$. From~\eqref{eq:uDG}, we have
\begin{equation}\label{eq:xDGx}
	\left(xD_{G_2}\right)^n(x)= \sum_{(k, \mu)\vdash n}  \sum_{Z} G_Z x x_{\mu_1}x_{\mu_2}\cdots x_{\mu_{n-1}}D_{G_2}^k(x).
\end{equation}
Note that
\begin{equation*}
x_0=D^0_{G_2}(x)=x,~x_1=D_{G_2}(x)=y,~x_2=D^2_{G_2}(x)=1,~x_i=D_{G_2}^i(x)=0 \quad\text{for $i\geq 3$}.
\end{equation*}
Recall that for $(k, \mu)\vdash n$, we have $k\in [n]$.
Then $x_{\mu_1}x_{\mu_2}\cdots x_{\mu_{n-1}}D_{G_2}^k(x)\neq 0$ if and only if $0\leq \mu_j\leq2$ for all $j\in [n-1]$ and $1\leq k\leq2$.
Thus
\begin{equation}\label{muform}
\mu=(1^{m_1}2^{m_2}0^{n-1-m_1-m_2}),
\end{equation}
where $m_1$ and $m_2$ are nonnegative integers.
Let $Z$ be a $k$-Young tableau of shape $(k,\mu)$, where $\mu$ is given by~\eqref{muform}.
As in the proof of Theorem~\ref{th:Cn},
we define $T=\rho(Z)$ to be the unique standard Young tableau
such that the sets of the letters in the $j$-th column in $Z$ and $T$ are the same for all $j$.
Then $Z$ has at most two columns.
Therefore, by using Proposition~\ref{grammar02:Snk}, we get
\begin{equation}
S_n(x) =\left. (xD_{G_2})^n (x) \right|_{y=1}
	=\sum_{(k, \mu)\vdash n}  \sum_{Z} G_Z x^{|\mu|_0+1},
\end{equation}
where the first summation is taken over all types $(k, \mu)$ of $n$, the second summation is taken over all $k$-Young tableaux of shape $(k,\mu)$ and the partitions $\mu$ have the form~\eqref{muform}.
In the same way as the proof of Theorem~\ref{th:Cn}, by using Lemma~\ref{lemmaGTZ}, we get~\eqref{SnxSYT}.
\end{proof}
\section{Concluding remarks}
In this paper, we present combinatorial expansions of $(c(x)D)^n$ in terms of inversion sequences as well as $k$-Young tableaux.
By introducing the $g$-index of Young tableau, we find that Eulerian polynomials, second-order Eulerian polynomials, Andr\'e polynomials and the generating polynomials of gamma coefficients of Eulerian polynomials can be expressed in terms of standard Young tableaux, which imply a deep connection among these polynomials.

\end{document}